\documentclass[10pt]{amsart}
\usepackage{amssymb}
\usepackage[headings]{fullpage}
\usepackage{amsfonts}
\usepackage{paralist}
\usepackage{color}

\usepackage[colorlinks,pagebackref=true,pdftex]{hyperref}

\makeatletter
\@addtoreset{equation}{section}
\def\theequation{\thesection.\@arabic \c@equation}

\def\theenumi{\@roman\c@enumi}
\def\@citecolor{blue}
\def\@linkcolor{blue}
\def\@urlcolor{blue}

\makeatother

\newtheorem{theorem}[equation]{Theorem}
\newtheorem{lemma}[equation]{Lemma}
\newtheorem{proposition}[equation]{Proposition}
\newtheorem{corollary}[equation]{Corollary}

\newtheorem{claim*}{Claim}
\newtheorem{question}[equation]{Question}

\newtheorem{innercustomlemma}{Lemma}
\newenvironment{customlemma}[1]
  {\renewcommand\theinnercustomlemma{#1}\innercustomlemma}
  {\endinnercustomlemma}

\theoremstyle{definition}
\newtheorem{remark}[equation]{Remark}

\newtheorem{eg}[equation]{Example}

\newtheorem{definition}[equation]{Definition}

\newtheorem{notn}[equation]{Notation}
\newenvironment{notation}[1][]{\begin{notn}[#1]\pushQED{\qed}}{\popQED
\end{notn}}

\newsavebox{\upperboundtheorem}
\usepackage{color}

\title{On the projective dimension of $5$ Quadric Almost complete intersections with low multiplicities }
 
\author{Sabine El Khoury} 
\address{Department of Mathematics, American University of Beirut, Beirut,
Lebanon.}
\email{se24@aub.edu.lb}  

\date {2,19,2018}

\begin{document}

 \begin{abstract} Let $S$ be a polynomial ring over an algebraic closed field $k$ and $ \mathfrak p =(x,y,z,w) $ a homogeneous height four prime ideal.   We give a finite characterization of the degree two component of ideals primary to $\mathfrak p$, with multiplicity $e \leq 3$. We use this result to give a tight bound on the projective dimension of almost complete intersections generated by five quadrics with $e \leq 3$. 
  \end{abstract}
\maketitle
\section{Introduction}
Let $S$ be the polynomial ring over an algebraically closed field $k$, and $x,y,z,w \in S$ four linearly independent linear forms. When  $\mathfrak p=(x,y)$ is a height $2$ prime ideal, Engheta shows there are two distinct types of $\mathfrak {p}$-primary ideals of multiplicity $2$, \cite[Proposition 11]{E07}.  As a consequence, he proves that if $I$ is an ideal generated by three cubics, then the projective dimension of $S/I$ is at most 36 \cite{E10}. This answers a specific case to Stillman's question:
\begin{question}(Stillman \cite[Problem 3.14]{PS09}) Is there a bound on the pd$(S/I)$ depending only on $d_1, \ldots, d_N $ and $N$, where $d_i =deg(f_i)$?
\end{question}
Furthermore, the bound found by Engheta is not optimal see \cite{FMP16} and  \cite{MS12}, as the expected projective dimension of three cubics is $5$. In \cite{HMMS16}, Huneke et.al. showed that for each $e \geq 3$ and for any $n \in \mathbb N$, there exists an ideal $J$ primary to a linear prime ideal $(x_1, \ldots x_h)$, with  $e(S/J)=e$ and pd$(S/J) \geq n$. Therefore, no natural extensions for Engheta's result on multiplicity $2$ primary ideals would be possible. In \cite{MM17}, Mantero and McCullough found a finite classification for $\mathfrak p$-primary ideals $J$ of multiplicities $3$ and $4$, by imposing an upper bound on the degrees of the generators of $J$. In particular, they find a finite classification of the linear, quadric and cubic generators of such $\mathfrak p$-primary ideals of multiplicities $3$ and $4$. They use their results to improve Engheta's estimates on three cubics, and show that the projective dimension is at most $5$, {\cite{MM171}. 
Stillman's question was proven recently by Ananyan and Hochster \cite{AH17}. They show the existence of a bound for any homogenous ideal $I$ generated by $N$ forms of degree at most $d$. However, the bounds they produce are very large even for ideals generated by quadrics or cubics. In a previous paper \cite {AH12}, they get a tighter bound for ideals generated by $N$ quadrics, but again not optimal. It remains open to find what are the best bounds for projective dimensions.  \\
When $\mathfrak p =(x,y,z)$, Huneke et.al. give a characterization of the degree $2$ component of an ideal $J$ primary to $\mathfrak p$ with e$(S/J) \leq 4$, see \cite{HMMS13}. By using this classification, they obtain a tight upper bound for the projective dimension generated by four quadrics. They show that if $I$ is generated by four quadrics, then pd$(S/I) \leq 6$. They also pose the following question:
\begin{question} (\label{HMMS}\cite[Question 10.2]{HMMS13} and \cite [Question 6.2]{HMMS131}) Let $S$ be a polynomial ring, and let $I$ be an ideal of $S$ generated by $n$ quadrics and having ht$(I)=h$. Is it true that pd$(S/I) \leq h(n-h+1)$?
\end{question} 
In our paper, we treat the case  when $\mathfrak p =(x,y,z,w)$. Our results, Propositions \ref{multiplicity2} and \ref{mult3},  offer a finite classification, in the sense of \cite{HMMS13}, for $\mathfrak p$-primary ideals $J$ with multiplicity $e(S/J) \leq 3$. Our characterization depends on the degree two component of $J$. We use  these classifications along with a result of \cite{HT17}, to answer Question \ref{HMMS} for almost complete intersections $I$ generated by $5$ quadrics with low multiplicities. We show that if $I$ is an almost complete intersection with multiplicity  $e(S/I) \leq 3$, then pd$(S/I) \leq 8$. \\ 
 Beyond the applications to Stillman's question, different characterization and structures of $\mathfrak p$-primary ideals have been given by algebraic geometers who study  vector bundles and multiple or nilpotent structures, see \cite {M92}, \cite{M04} and \cite{V09}  for instance.  \\ 
The rest of the paper is divided in the following way. In section $2$, we collect results that are useful to us and set our notation. In Sections $3$ and $4$, we give our characterization of height four primary ideals with multiplicities $2$ and $3$ respectively. These primary ideals can occur as components of the unmixed part of our $5$ quadrics ideal or a direct link to this ideal. In section $4$, we apply our previous results to prove that the projective dimension of an almost complete intersection generated by $5$ quadrics with $e \leq 3$ is at most $8$. Appendix $A$ contains a list of primary, unmixed ideals that are essential to our sections.

\section {Preliminaries}

In this section, we set our notations and collect results that are useful to our theorems. 

\subsection*{Unmixed ideals and multiplicity} We use the associativity formula to compute the multiplicity of an ideal

\begin{proposition}[Associativity Formula] \cite [Theorem 11.2.4]{HS06}If J is an ideal of S, then 
$$e(S/J) = \displaystyle \sum_{\tiny \begin{array}{ccc} J \subset \mathfrak p\\ \mbox{ht}(\mathfrak p)=\mbox{ht}(J) \end{array}}   e(S/\mathfrak p)\lambda(S_{\mathfrak p}/J_{\mathfrak p})$$
where $e(S)$ denotes the multiplicity of a graded ring $S$, and $\lambda(M)$ denotes the length of an $S$-module $M$.
\end{proposition}

An ideal $J$ of height $h$ is unmixed if ht$(\mathfrak p) = h$ for every $\mathfrak p \in$ Ass$(S/J)$. The unmixed part of $J$, denoted $J^{un}$, is the intersection of all the components of J of minimum height. We have $J \subset J^{un}$, and $e(S/J) = e(S/J^{un})$. We follow  the notation of \cite{E07}: if $J$ is an unmixed ideal, we say it is of type
$\langle e_1,\ldots ,e_m; \lambda_1,\ldots, \lambda_m \rangle$. If $J$ has $m$ associated prime ideals $\mathfrak p_1, \mathfrak p_2 \ldots \mathfrak p_m$, then $e(S/J)= \sum_{i=1}^me_i \lambda_i$, where $e_i= e(S/ \mathfrak p)$ and $\lambda_i= \lambda (S_{\mathfrak p_i}/J_{\mathfrak p_i})$. We also get

\begin{lemma} \label {equalideals} \cite [lemma 8] {E07} Let $J \subset S$ be an unmixed ideal. If $I \subset S$ is an ideal containing $J$ such that $ht(I) = ht(J)$ and $e(S/I)=e(S/J)$, then $J=I$. 
\end{lemma}

\subsection*{Linkage}
Two ideals $J$ and $K$ in a regular ring $R$ are said to be {\it linked} $J \sim K$, if there exists a regular sequence $\underline{\alpha}= \alpha_1, \ldots \alpha_g$ such that $K = (\underline{\alpha}):J$ and $J = (\underline{\alpha}):K$. Notice that the definition forces $J$ and $K$ to be unmixed, and $(\underline{\alpha}) \subset J \cap K$.\\
 
 The following results on linkage are needed.
\begin{theorem} \label{PS} (Peskine-Szpiro \cite{PS74}) Let $J$ be an unmixed ideal of $S$ of height $g$. Let $\underline{\alpha}= \alpha_1, \ldots, \alpha_g $ be a regular sequence in $J$, and set $K = (\underline \alpha):J $. Then one has
\begin{enumerate} 
\item $J = (\underline \alpha):K $, that is, $J \sim K $ via $\underline{\alpha}$;
\item $S/J$ is CM if and only if $S/K$ is CM;
\item $e(S/J)+e(S/K) = e(S/(\underline \alpha))$
\end{enumerate}
\end{theorem}

If $(\underline{\alpha})$ is a regular sequence of maximal length in an ideal $I$, then $(\underline{\alpha}):I=(\underline{\alpha}):I^{un}$ that is $I^{un}$ is linked to $(\underline{\alpha}):I$.

\begin{lemma} \cite[Lemma 2.6]{EB05}  \label{linked-ideals=pdim} Let $R$ be a Gorenstein local ring and let $I \subset R$ be an unmixed ideal. All ideals which are linked to $I$ have the same (finite or infinite) projective dimension. 
\end{lemma}

\begin{lemma} \label{pdim-linkage}  \cite[Theorem 7]{E07} 
Let $J$ be an almost complete intersection ideal of $S$. If $K$ is an ideal linked to $J^{un}$, then $pd (S/J) \leq pd(S/K)+1$
\end{lemma}

\subsection* {Basic Results} 
\begin{theorem} \label{SN} (Samuel \cite{S51}, Nagata \cite[Theorem 40.6]{N75})
Let $J$ be a homogeneous unmixed ideal of $S$. If $e(S/J)=1$,  then J is generated by
ht$(J)$ linear forms.
\end{theorem}

A  homogeneous ideal $I$ is called degenerate if $I$ contains at least one linear form, otherwise $I$ is said to be non-degenerate. We need the next result on a lower bound for the multiplicity of non-degenerate prime ideals.
\begin{proposition} \label{mult-height} \cite[Corollary 18.12]{H92} Let $\mathfrak p$ be a homogeneous prime ideal of $S$. If $\mathfrak p$ is a non-degenerate prime ideal, then $e(S/\mathfrak p) \geq ht (\mathfrak p)+1$
\end{proposition}

\begin{corollary}\label{prime} Let $\mathfrak p$ be a homogeneous prime ideal of $S$ of height four. If $e(S/ \mathfrak p)=2$, then there exists linear forms $x$, $y$, $z$  and a quadric $q$ such that $\mathfrak p = (x,y,z,q)$
\end{corollary}

The following consequence of a result of Ananyan and Hochster \cite [Lemma 3.3]{AH12} is essential to our paper.

\begin{lemma}\label{extend1}\cite[Lemma 3.3]{AH12} Let $f_1 \ldots f_t$ be a regular sequence of forms in $S$ and set $A=k [f_1,\ldots, f_t]$. Then for any ideal $I$ of $S$ extended from $A$, one has $\mbox{pd}(S/I) \leq t$. More generally for any finite $S$-module $M$ presented by a matrix with entries in $A$, one has pd$(M) \leq t$. In particular, any ideal $I$ whose generators can be written in terms of at most $t$ variables satisfies pd$(S/I)\leq t$. 
\end{lemma}

The following Corollary of Lemma  \ref {extend1} applies to almost complete intersections.
\begin{corollary} \label{extend}  \cite[Corollary 3.9]{HMMS13} Let $I$ be an almost complete intersection of height $h$. Suppose $I^{un}$ is extended from  $A =k[f_1, \ldots, f_t]$ where $f_1, \ldots f_t$ form a regular sequence then $pd(S/I) \leq max(h+2,t)$
\end{corollary}

\begin{proposition}\label{coker} \cite[Proposition 3.7]{HMMS13} Let $I$ be an almost complete intersection of height $h$. Let $F_{\bullet}$ be the minimal free resolution of the unmixed part $I^{un}$ of $I$. Let $\partial_i$ denote the $i^{th}$ differential in $F_{\bullet}$. Then
$$\mbox{pd}(S/I) \leq max \{h+2, \mbox{pd}(Coker(\partial^*_{h+1}))$$ 
\end{proposition}

\begin{lemma}\label{intersection}\cite[Lemma 3.6]{HMMS13} Suppose $I_1, I_2$ are two ideals. Then $$\mbox{pd}(S/(I_1 \cap I_2)) \leq \mbox{max}\{\mbox{pd}(S/I_1),\mbox{pd}(S/I_2), \mbox{pd}(S/(I_1+I_2)-1) \}$$
\end{lemma}

\section {$(x,y,z,w)$-primary of multiplicity 2}

In this section, we give a characterization of the degree $2$ component of an ideal primary  to the linear prime $\mathfrak{p}=(x,y,z,w)$ of multiplicity $2$, similar to the characterization found by C.~Huneke et~al.  \cite [Proposition 4.3]{HMMS13}. Before we give our main result, we prove lemmas on matrices of linear forms that are essential to the characterization of these ideals.  For any $m \times n$ matrix $M$  with $(m \leq n)$, we denote $I_j(M)$ to be the ideal generated by the $j \times j$ minors of $M$. $I_j(M)$ is unchanged by linear row and column operations. A matrix $M$ is $1$-generic if it does not have a zero entry called a generalized zero, after row/column operations. If $M$ is $1$- generic, then the ideal $I_{m}(M)$ generated by the maximal minors of $M$ is prime and of codimension $n-m+1$,  \cite[Theorem 6.4]{E05}.

\begin{lemma} \label {lemmaI_2} Let $M= \begin{pmatrix}a&b&c&d\\ e&f&g&h \end{pmatrix}$ where $a, \ldots h$ are linear forms, such that ht$(I_2(M))=1$. Then, after row and column operations, $M$ is one of the following
\begin{enumerate}
\item \label {anot0} $M=\begin{pmatrix}a&0&0&0\\ e&f&g&h \end{pmatrix}$ with $a \neq 0$. 
\item \label {afbe} $M= \begin{pmatrix}a&b&0&0\\ e&f&0&0 \end{pmatrix}$ with $af-be \neq 0$.
\item \label {ht(a,b,e)=3} $M= \begin{pmatrix}a&b&0&0\\ e&0&b& 0\end{pmatrix}$ with ht$(a,b,e)=3$. 
\end{enumerate}
\end{lemma}

\begin{proof} Since ht$(I_2(M))=1$ then $M$ is not $1$-generic \cite[Theorem 6.4]{E05} and has a generalized zero, say $d=0$. If ht$(a,b,c)=1$ then $M$ has the form of \ref{anot0}. If  ht$(a,b,c)=2$, then we may  assume  $M= \begin{pmatrix}a&b&0&0\\ e&f&g&h \end{pmatrix}$. If ht$(e,f,g,h)\geq 3$ then, since $I_2(M)=(af-be,ag, ah,bg, bh)= (a,b) \cap (af-be, g, h)$, we get ht$(I_2(M))>1$. Hence, we may assume ht$(e,f,g,h)=2$ (when it is equal to $1$, we $M$ has the form of \ref{anot0}). So the only two possibilities for $M$ would be
$$M= \begin{pmatrix}a&b&0&0\\ e&f&0&0 \end{pmatrix} \hspace{1cm} \mbox{or} \hspace{1cm}    M = \begin{pmatrix}a&b&0&0\\ e&0&f&0 \end{pmatrix}$$
In the latter case, we have $I_2(M)= (be, af,bf)$ which implies $f \in (b)$. This puts us in case \ref{ht(a,b,e)=3}. \\
Finally, when ht$(a,b,c)=3$, then $I_2(M)= (af-be,  ag-ce, bg-cf, ah, bh, ch )=(af-be, ag-ce, bg-cf,h) \cap (a,b,c)$. This implies that ht$\left(I_2 \begin{pmatrix}a&b&c\\ e&f&g \end{pmatrix}, h \right)=1$. Hence, by \cite[Lemma 4.1]{HMMS13}, we get ht$(e,f,g,h)=1$, which puts us in \ref{anot0}. \end{proof}

\begin{lemma} \label{lemmaI_3(M)} Let $M =\begin{pmatrix} a&b&c&d\\
e&f&g&h\\k&l&m&n \end{pmatrix}$  be a $3 \times 4$ matrix of generic forms. Suppose $ht(I_3(M))=1$ then, after row and column operations, $M$ has one of the following form
\begin{enumerate}
 \item  \label {ht(a,b,c)=1}    $M= \begin{pmatrix} a&0&0&0\\
e&f&g&h\\k&l&m&n \end{pmatrix}$ with $a \neq 0$ and ht$\left(I_2\begin{pmatrix} f&g&h\\ l&m&n \end{pmatrix}\right)>0$. \\
\item \label{ht(a,b)=2} $M=\begin{pmatrix} a&b&0&0\\
e&f&0&0\\k&l&m&n \end{pmatrix}$ with $af-be \neq 0$ and ht$(m,n) >0$.
\item \label{ht(a,e,g)=3} $M=\begin{pmatrix} a&b&0&0\\
e&0&b&0\\k&l&m&n \end{pmatrix}$ with $n \neq0$ and ht$(a,b,e)=3$.
\item \label{ht(a,b)=2,nnot0} $M=\begin{pmatrix} a&b&0&0\\
e&f&n&0\\k&l&0&n \end{pmatrix}$ with $n \neq 0$ and ht$(a,b)=2$.
\item \label{det} $M=\begin{pmatrix} a&b&c&0\\
e&f&g&0\\k&l&m&0 \end{pmatrix}$  with det$\begin{pmatrix} a&b&c \\
e&f&g\\k&l&m \end{pmatrix} \neq 0$.
\item\label{form} The form of $M$ is determined by at most four variables.
\end{enumerate}
\end{lemma}

\begin{proof} Since ht$(I_3(M))=1$ then $M$ is not $1$- generic \cite[Theorem 6.4]{E05} and has a generalized zero. So we may assume that $M$ has the form of
$M =\begin{pmatrix} a&b&c&0\\
e&f&g&h\\k&l&m&n \end{pmatrix}$.
 If ht$(a, b, c) = 1$, then $M$ has the form of \ref{ht(a,b,c)=1}. Suppose ht$(a,b,c) \geq 2$ and consider $M_1= \begin{pmatrix} e&f&g&h\\k&l&m&n \end{pmatrix}$.  $M_1$ is not $1$-generic and has a generalized zero, or else ht$(I_3( M))\neq1$. Hence,  after a column operation and a linear change of variables, $M$ is equal to either  $M =\begin{pmatrix} a&b&c&0\\
e&f&g&0\\k&l&m&n \end{pmatrix}$ or $M =\begin{pmatrix} a&b&c&0\\
e&f&0&h\\k&l&m&n \end{pmatrix}$.  So for the rest of the proof, we assume ht$(e,f,g,h) \leq 3$.\\

 \underline{Case 1.} Suppose $M =\begin{pmatrix} a&b&c&0\\
e&f&g&0\\k&l&m&n \end{pmatrix}$. If $n=0$ and det$(N) \neq 0$ where $N$ is the matrix obtained from $M$ by deleting the last column, then $M$ has the form of \ref{det}. We assume $n \neq 0$, and we write $I_3(M)= ($det$ N, n \delta_1, n\delta_2, n\delta_3)$  where $(\delta_1, \delta_2,\delta_3)=I_2(M')=I_2 \begin{pmatrix}  a&b&c\\e&f&g \end{pmatrix}$. Hence $I_3(M)=(\delta_1, \delta_2,\delta_3) \cap (n, det N)$. We also write $det(N)=k(bg-fc)-l(ag-ec)+m(af-be)= k \delta_1-l\delta_2+m\delta_3$. Since ht$(I_3(M))=1$, then either ht$(\delta_1, \delta_2, \delta_3)=$ht$(I_2(M'))=1$ or det$N \in (n)$ (when det$(N) \in k,l $ or $m$, we get back to one of the two cases). If ht$(I_2(M'))=1$ then by \cite[Lemma 4.1]{HMMS13},  we get either \ref {ht(a,b,c)=1}, \ref {ht(a,b)=2}, or  \ref {ht(a,e,g)=3}. If det$(N) \in (n)$, then we get either $\delta_1, \delta_2,\delta_3 \in (n)$ or $k,l,m \in (n)$ or $\delta_1, \delta_2, m \in (n)$ or $\delta_1, l, m \in (n)$, and the remaining cases are identical. The first case implies ht$(I_2(M'))=1$, which was treated above. The second case puts us in \ref {ht(a,b,c)=1}, the third one in either \ref{ht(a,b,c)=1} or \ref{ht(a,b)=2,nnot0}, and the last one in either \ref{ht(a,e,g)=3} or \ref{ht(a,b)=2,nnot0}.

 \underline{Case 2.} Suppose $M =\begin{pmatrix} a&b&c&0\\ e&f&0&g\\k&l&m&n \end{pmatrix}$. In this case, we study the two possible values of  ht$(a,b,c)$.
 \begin{enumerate}
\item Suppose ht$(a,b,c)=2$, then either $M =\begin{pmatrix} a&b&0&0\\
e&f&g&0\\k&l&m&n \end{pmatrix}$ or  $M =\begin{pmatrix} a&b&0&0\\
e&0&f&g\\k&l&m&n \end{pmatrix}$. The former case was already treated in Case 1, so we discuss the latter case.\\
 $\bullet$  If ht $(e, f, g) \leq 1$ then we are back to case \ref{ht(a,b,c)=1}. \\
 $\bullet$  If ht$(e, f, g)=2$, then $M$ have one of the following forms:
 $$ a)M =\begin{pmatrix} 
a&b&0&0\\ 0&e&f&0\\k&l&m&n \end{pmatrix} \mbox{or} \hspace{0.3cm} b) M =\begin{pmatrix} 
a&b&0&0\\ 0&0&e&f\\k&l&m&n \end{pmatrix}  \hspace{0.3cm} \mbox{with ht$(a,b)=$ht$(e,f)=2$.}$$
-- If $M$ has the form of $a)$, then we get back to Case 1. \\
-- Suppose $M$ has the form of $b)$. We have $I_3(M)= \left( e(al-bk), f(al-bk), a(en-fm), b(en-fm)\right)= (a,b) \cap (e,f) \cap (al-bk, en-fm)$ with height  equal to $1$. We only study the cases when $I_3(M) \subset (a)$ or $I_3(M) \subset (al-bk)$, since the other ones are obtained in a similar manner after a linear change of variables. If $I_3(M) \subset (a)$, then $ebk, fbk, ben, bfm \in (a)$. This implies that $k \in (a)$, and either $e,m \in (a)$ or $m,n \in (a)$. We may assume $k=0$ after a row operation. If $e,m \in (a)$, then we get back to \ref{ht(a,b)=2,nnot0}. If $m, n \in (a)$, then we get back to  \ref{ht(a,e,g)=3} after row/column operations and a linear change of variables. When $I_3(M)=(e(al-bk), f(al-bk), a(fm-en), b(fm-en)) \subset (al-bk)$, then $(en-fm) \in (al-bk)$  and the form of $M$ is defined by $4$ variables which puts us in \ref{form}.\\
$\bullet$  Suppose ht$(e, f, g)=3$ and write  $M =\begin{pmatrix}  a&b&0&0\\ e&0&f&g\\k&l&m&n \end{pmatrix}$. We have $I_3(M)= (afl-b(em-fk), agl-b(en-gk), a(fn-gm), b(fn-gm))=(a,b) \cap (fn-gm, afl-b(em-fk), agl-b(en-gk))$ with height $1$. We may assume that  ht$(k,l,m,n) \geq 3$, since otherwise this would be identical to the case when ht$(e, f, g)=2$ . If $I_3(M) \subset (a)$, then $(em-fk), (en-gk), (fn-gm) \in (a)$. This implies that ht$\left(I_2\begin{pmatrix} e&f&g\\k&m&n \end{pmatrix} \right)=1$. By \cite[Lemma 4.1]{HMMS13}, we get either ht$(e,f,g) \leq 2$ or ht$(k,m,n) =1 $, which is impossible. If $I_3(M) \subset (b)$, then  $l, fn-gm \in (b)$. We may assume $l=0$ after a row operation and we are back to Case 1. If $I_3(M) \not\subset (a)$ nor in $(b)$, then ht$((fn-gm, afl-b(em-fk), agl-b(en-gk))=1$. We get  ht$\left(I_2\begin{pmatrix} e&f&g\\k&m&n \end{pmatrix}, fl, gl \right)=1$. Again, by \cite[Lemma 4.1]{HMMS13}, we get ht$(e,f,g) \leq 2$ or ht$(k,m,n) =1$ which is impossible. 

\vskip 0.3cm
\item  Suppose ht$(a,b,c)=3$, and $M =\begin{pmatrix} 
a&b&c&0\\ e&f&0&g\\k&l&m&n \end{pmatrix}$. Again, we assume that ht$(e,f,g)=3$ and ht$(k,l,m,n) \geq 3$, since otherwise these cases would be identical to whenever ht$(a,b,c)\leq 2$. We have $I_3(M)= (afm-bem+ c(el-fk), a(fn-gl)-b(en-gk), agm+c(en-gk), bgm+c(fn-gl))$. Since ht$(I_3(M))=1$, then $I_3(M)$ is contained in a prime ideal $P$ of height $1$. Every prime ideal of height $1$ in a UFD is principal, then $P$ is generated by a either a degree $1$, a degree $2$ or a degree three element. If $P$ is generated by a degree one element $u$, then $u$ would be one of the variables $a, b, c, e,f ,g, k, l, m$ or $n$. The cases when $I_3(M) \subset (a)$,$(b)$, $(e)$ or $(f)$ are identical, so we treat one of them. If $I_3(M) \subset (a)$, then $em, gm, (el-fk), (en-gk), (fn-gl) \in (a)$. It implies that ht$\left( I_2\begin{pmatrix}  e&f&g\\k&l&n \end{pmatrix}, em, gm\right)=1$. By \cite[Lemma 4.1]{HMMS13}, we get ht$(e,f,g) \leq 2$ or ht$(k,m,n)=1$ which is impossible. Again the cases of $I_3(M) \subset (c)$ or $(g)$ are identical.   If $I_3(M) \subset (c)$, then $fm, em, gm, (fn-gl), (en-gk) \in (c)$. Hence either  $m, n, g \in (c)$ and $M$ has the form of \ref{ht(a,b)=2,nnot0}, or $k,l,m \in (n)$  which is impossible in this case (we actually get back to \ref {ht(a,b,c)=1} but we assumed that $ht(k,l,m,n) \geq 3$).
 If $I_3(M) \subset (k), (l), (m)$ or $(n)$, then we again get back to the case when ht$(a,b,c)\leq 2$. Suppose $I_3(M)$ is contained in a prime ideal generated by a degree two form $u$. Since every term in $I_3(M)$ is a product of three linear forms of $M$, we may assume $u \in I_2(M)$. We either get ht $\left( I_2\begin{pmatrix}  e&f&0&g\\k&l&m&n \end{pmatrix}\right)= 1$ or ht $\left( I_2\begin{pmatrix}  a&b&c&0\\k&l&m&n \end{pmatrix}\right)= 1$ which are both impossible by Lemma \ref{lemmaI_2}, or we obtain cases that lead to either ht$(e,f,g)$ or ht$(k,l,m,n)$ less than $3$ which is also a contradiction.
 Finally, if  $u$ is a form of degree three, then $u$ must be one of the generators of $I_3(M)$. We again get to ht$(e,f,g) \leq 3$ or ht$(k,l,m,n) \leq 3$. \end{enumerate}

\end{proof} 
\begin{remark}\label{4variables}We record the matrices of \ref{lemmaI_3(M)}{\color{blue}-}\ref{form} as details are needed to complete the proof of Proposition \ref{multiplicity2}:
\begin{enumerate}
\item $M =\begin{pmatrix} a&b&0&0\\ 0&0&k&l\\k&l&a&b \end{pmatrix}$ with ht$(al-bk)=1$.
\item $M =\begin{pmatrix} a&b&0&0\\ 0&0&k&a\\k&l&l&b \end{pmatrix}$ with ht$(al-bk)=1$.
\end{enumerate}
\end{remark}

We then get to our first result on the characterization of $(x,y,z,w)$-primary ideals.

\begin{proposition} \label{multiplicity2} Let $J$ be an $(x,y,z,w)$-primary with $e(S/J) =2$ then one of the following holds:

\begin{itemize} \item[I-] If $J$ is degenerate then 
\begin{enumerate}
\item \label{linearform1} $J = (x,y,z,w^2)$.
\item \label{linearform2} $J= (x,y)+(z,w)^2+ (az+bw )$, ht$(x,y,z,w,a,b) =6$.
\item \label{linearform3} $J=(x)+ (y,z,w)^2+ (by+cz+dw, ey+fz+gw)$ with ht$\left(x,y,z,w,I_2\begin{pmatrix}b&c&d\\ e&f&g \end{pmatrix}\right) \geq 6$.
\item \label{vars21} All quadrics in $J$ can be written in terms of at most of $8$ variables.
\end{enumerate}

\item[II-]  If $J$ is non-degenerate then 
\begin{enumerate}
\item \label{I3} $J = (x, y, z,w)^2+ (ax+by+cz+dw, ex+fy+gz+hw, kx+ly+mz+nw)$  with ht$(x,y,z,w,I_3(M)\geq 6$, where $M=\begin{pmatrix} a&b&c&d\\
e&f&g&h\\k&l&m&n \end{pmatrix}$.
\item \label{I3=1}$J = (x,y,z,w)^2+(ax+by, ex+fy+nz, kx+ly+nw, (af-be)w-(al-kb)z)$ with ht$(a,b)=2$ and $n \neq 0$.
\item \label{I3=12} $J = (x,y,z,w)^2+(ax+by, ex+bz, ey-az, kx+ly+mz+nw)$ with ht$(a,b,e)=3$  and $n \neq 0$.
\item \label{vars2} All quadrics in $J$ can be written in terms of at most of $8$ variables.
\end{enumerate}

\end{itemize}

\end{proposition}

\begin{proof} The proof of the first three cases goes along the same line as the first cases of \cite [Proposition 4.3]{HMMS13}. Since $e(S/J)=2$, we have $(x,y,z,w)^2 \subset J \subset (x,y,z,w)$. If $J$ contains a linear form, then $J$ can be written as $J = (x)+(J')$ where $J'$ is $(y, z, w)$-primary with multiplicity $2$. By \cite [Proposition 4.3]{HMMS13}, the authors obtained $4$ different cases for $J'$, and hence $J$ will have the form of  {\color{blue}I-}\ref{linearform1},  {\color{blue}I-}\ref{linearform2},  {\color{blue}I-}\ref{linearform3} or  {\color{blue}I-}\ref{vars21}. If $J$ does not contain a linear form, then since $e(S/(x,y,z,w)^2)=5$, there exists three generators of the form $ax+by+cz+dw, ex+fy+gz+hw$ and $kx+ly+mz+nw$. We may assume that these generators are linearly independent or else $e(S/J) >2$. By lemma \ref {lemma1}, the ideal $(x, y, z,w )^2+ (ax+by+cz+dw, ex+fy+gz+hw, kx+ly+mz+nw)$ is unmixed with multiplicity $2$, and hence equal to $J$ if ht$\left(x,y,z,w,I_3\begin{pmatrix} a&b&c&d\\
e&f&g&h\\k&l&m&n \end{pmatrix}\right) \geq 6$. This puts us in case  {\color{blue}II-}\ref{I3}. We may suppose that ht$(I_3(M))+(x,y,z,w))/(x, y, z, w)=1$. By lemma $\ref{lemmaI_3(M)}$, $M$  can take six different forms. When $M$ has the form of \ref{lemmaI_3(M)}{\color{blue}.}\ref{ht(a,b,c)=1}, we have $ax \in J$ but $a \notin (x,y,z,w)$ and $J$ is $(x,y,z,w)$-primary then $x \in J$ and $J$ contains a linear form. When $M$ has the form of \ref{lemmaI_3(M)}{\color{blue}.}\ref{ht(a,b)=2}, we have $(af-be)y \in J$ and $(af-be) \notin (x,y,z,w)$ hence $y \in J$, also a contradiction. Similarly, when $M$ has the form of \ref{lemmaI_3(M)}{\color{blue}.}\ref{det}, we  get $((af-be)(am-kc)- (al-kb)(ag-ce))z=a(det(N))z \in J$, with $N$ obtained from $M$ by deleting the $4^{th}$ column.  However, $adet(N) \notin (x,y,z,w)$ hence $z \in J$, which is a contradiction as well. Suppose $M$ has the form of \ref{lemmaI_3(M)}{\color{blue}.}\ref{form}. By Remark \ref{4variables}{\color{blue}-i}, either  $J =(x,y,z,w)^2+(ax+by, kz+lw, ly+az, kx+bw, bz -lx, aw -ky) $ by Lemma \ref{lemma40} or $J$ contains a linear form. When $(x,y,z,w)^2+(ax+by, kz+aw, kx+ly+lz+bw)$ in Remark \ref{4variables}{\color{blue}-ii}, we get $(al-bk)(z+y)\in J$ and $z+y\in J$ . Hence $J$ contains a linear form, which is impossible. We conclude that, when $M$ has the form of \ref{lemmaI_3(M)}{\color{blue}.}\ref{form}, all quadrics are written in terms of $8$ variables which puts us in case {\color{blue}II-}\ref{vars2}. The remaining cases for $M$ are \ref{lemmaI_3(M)}{\color{blue}.}\ref{ht(a,e,g)=3} and \ref{lemmaI_3(M)}{\color{blue}.}\ref{ht(a,b)=2,nnot0}. In the latter case, we get $J = (x,y,z,w)^2+ (ax+by, ex+fy+nz, kx+ly+nw)$ and $an((al-bk)z-(af-be)w) \in J$. Since $an \notin (x,y,z,w)$, then $(al-bk)z-(af-be)w \in J$. Consider the following ideal $K=(x,y,z,w)^2+ (ax+by, ex+fy+nz, kx+ly+nw, (al-bk)z-(af-be)w)$ with ht$\left(x,y,z,w,n,I_2\begin{pmatrix} a&b\\ e&f \\ k&l\end{pmatrix}\right)=7$, or else $K$ contains a linear form. By lemma \ref {lemma2}, $K$ is unmixed of multiplicity $2$ and is equal to $J$, leaving us in case {\color{blue}II-}\ref{I3=1}. When $M$ has the form of \ref{lemmaI_3(M)}{\color{blue}.}\ref{ht(a,e,g)=3}, we have $ax+by$ and $ex+bz \in J$. We get, $eby-abz \in J$, since $b \notin (x,y,z,w)$ and $J$ is primary then $ey-az \in J$. We have  ht$(a,b,e)=3$, or else $J$ contains a linear form. By Lemma  \ref{lemma3}, the ideal $(x,y,z,w)^2+(ax+by, ex+bz, ey-az,kx+ly+mz+nw)$ is unmixed of height $4$ and multiplicity $2$, hence it is equal to $J$. This  leaves us in {\color{blue}II-}\ref{I3=12}.  
\end{proof}

\section {$(x,y,z,w)$-primary of multiplicity 3}
In this section, we give a characterization of the degree $2$ component of an ideal primary  to the linear prime $\mathfrak{p}=(x,y,z,w)$ of multiplicity $3$. Before stating our result, we again prove a lemma on matrices with linear forms.
We start with the following:

\begin{lemma} \label{I_2(M_1)}Let $M= \begin{pmatrix}a&b&c&d\\e&f&g&h\end{pmatrix}$ with ht$(I_2(M))=2$ then $M$ is one of the following:
\begin{enumerate}
\item \label{ab} $M= \begin{pmatrix}a&b&0&0\\e&f&g&h\end{pmatrix}$ with ht$(a,b)=2$, ht$(g,h)>0$ and ht$(e,f,g,h) \geq 2$.
\item \label {h=0} $M= \begin{pmatrix}a&b&c&0\\e&f&g&0\end{pmatrix}$ with ht$(a,b,c)=$ht$(e,f,g) =3$.
\item \label{c} $M= \begin{pmatrix}a&b&c&0\\e&f&0&c\end{pmatrix}$ with ht$(af-be, c)=2$, ht$(a,b,c)=$ht$(e,f,c)=3$.
\item \label {bc} $M= \begin{pmatrix}a&b&c&0\\e&0&b&c\end{pmatrix}$ with ht$(a,b,c,e)=4$.
\end{enumerate}
\end{lemma}

\begin{proof} Since ht$(I_2(M))=2$, then $M$ is not $1$-generic and has a generalized zero. So we may assume that $d=0$ and $M= \begin{pmatrix}a&b&c&0\\e&f&g&h\end{pmatrix}$. We get $I_2(M) =(af-be, ag-ec, bg-fc, ah, bh, ch)=(a,b,c) \cap \left(h, I_2\begin{pmatrix} a&b&c\\ e&f&g \end{pmatrix}\right)$. Hence the following cases may occur:
\begin{itemize}
\item ht$(a,b,c)=2$, which puts us in \ref{ab}
\item $h=0$, which is case \ref{h=0}
\item ht$(a,b,c)=3$, ht$(e,f,g,h) \geq 3$ and $I_2(M)$ is contained in a prime of height two.
\end{itemize} 

We study the last case in detail. The primes of height two that contain $I_2(M)$ are either obtained from two linear forms or from one linear form and one quadric. The linear forms are namely coming from $a,b,c,e,f,g,h$  and the quadric one from $I_2 \begin{pmatrix} a&b&c\\ e&f&g \end{pmatrix}$. Since ht$(a,b,c)=3$ and ht$(e,f,g,h) \geq 3$, we only study the cases when the primes containing $I_2(M)$ are $(a,c)$ or $(h, af-be)$. All other cases lead us back to these two cases, either by a linear change of variables or by studying the generators of $I_2(M)$. For instance, if we are studying the case when $I_2(M) \subset (a,e)$ then $bg, fc, bh, ch \in (a,e)$. Since ht$(a,b,c)=3$ and ht$(e,f,g,h) \geq 3$, then either $b$ or $c$ are in $(e)$. We may assume $c=e$. We show below that if $I_2(M) \subset (a,c)$ then either $e=c$ or $e=0$.  Hence, since ht$(I_2(M))=2$, then the case when $I_2(M) \subset (a,e)$ is be equivalent to the case when $I_2(M) \subset (a,c)$.\\
Suppose  $I_2(M) \subset (a,c)$. Since ht$(a,b,c)=3$, we get $e,g,h \in (a,c)$. We assume $h \neq 0$ or else we are in  \ref{h=0}. The cases when $h=a$ and $h=c$ are identical so we study one of them in detail. 
Suppose $h=a$, then we may assume that $g=0$ by row/column operations and $e=c$. Hence we get to \ref{bc} by row/column operations. We note, that when $h=c$, we get $e=0$ and $g=a$. We again obtain \ref{bc} by row/column operations.
 When $I_2(M) \subset (af-be, h)$, we have $bg-cf, ag-ce \in (af-be, h)$. This puts us  either in \ref{ab} or in \ref{c} when $g =0$ and $c\in h$. We note that we consider ht$(a,b,c)=$ht$(e,f,c)=3$ in \ref{c} and ht$(a,b,c,e)=4$ in \ref{bc}, or else we go back to \ref{ab}.
\end{proof}

 \begin{lemma} \label{lemmaI_3(M)=2} Let $M =\begin{pmatrix} a&b&c&d\\
e&f&g&h\\k&l&m&n \end{pmatrix}$ with ht$(I_3(M))=ht(I_2(M))=2$. Suppose ht$(a,b,c) \geq 2$ and $d=0$, then $M$ is one of the following:

\begin{enumerate}
\item \label {abab} $M = \begin{pmatrix} a&b&0&0\\ 0&a&b&0 \\ k&l&m&n \end{pmatrix}$ with $n \neq 0$, ht$(a,b,k,l,m,n)\geq 5$.
\item \label {abab1} $M=\begin{pmatrix} a&b&0&0\\ 0&0&a&b \\ k&l&m&n \end{pmatrix}$ with ht$(a,b,k,l,m,n) \geq 5$.
\item \label {bg} $M = \begin{pmatrix} a&b&c&0\\ e&f&g&0 \\ k&0&f&g \end{pmatrix}$ or $M = \begin{pmatrix} a&b&c&0\\ e&f&g&0 \\ k&0&g&f \end{pmatrix}$ with ht$(a,e,f,g,k)=5$ and $b,c \in (f,g)$ (one of $b$ or $c \neq 0$). 
 \item \label {form1} The form of $M$ is determined by at most four variables.
  \end{enumerate}

\end{lemma}

\begin{proof}  Since ht$(I_3(M))=\mbox{ht}(I_2(M))=2$, we first show that ht$(I_2(M_i))=2$ for all $i$ where $M_1=\begin{pmatrix} a&b&c&d\\e&f&g&h \end{pmatrix}$, $M_2=\begin{pmatrix} a&b&c&d\\ k&l&m&n \end{pmatrix}$, $M_3=\begin{pmatrix} e&f&g&h\\ k&l&m&n \end{pmatrix}$.  For that, if there exists an $i$ such that ht$(I_2(M_i))=3$, then ht$(I_2(M) )\geq $ ht$(I_2(M_i))=3$ which is impossible. If there exists an $i$ such that ht$(I_2(M_i))=1$, then $M_i$ is one of Lemma \ref{lemmaI_2} which implies that ht$(I_3(M))=1$. \\
Since ht$(I_2(M_1))=2$, then $M_1$ is one of the cases of Lemma \ref{I_2(M_1)} and $M$ takes one of the following forms:\\
 1) $M = \begin{pmatrix} a&b&0&0\\ e&f&g&h \\ k&l&m&n \end{pmatrix}$, 2) $M = \begin{pmatrix} a&b&c&0\\ e&f&g&0 \\ k&l&m&n \end{pmatrix}$ , 3) $M = \begin{pmatrix} a&b&c&0\\ e&f&0&c \\ k&l&m&n \end{pmatrix}$ or  4) $M = \begin{pmatrix} a&b&c&0\\ e&0&b&c \\ k&l&m&n \end{pmatrix}$ \\
\underline {Case 1):} -Suppose $M = \begin{pmatrix} a&b&0&0\\ e&f&g&h \\ k&l&m&n \end{pmatrix}$. Since ht$(I_2(M_3))=2$, then $M_3$ has a generalized zero and we get two forms for the matrix $M$: either $M = \begin{pmatrix} a&b&0&0\\ e&f&g&0 \\ k&l&m&n \end{pmatrix}$ or $M = \begin{pmatrix} 0&0&a&b\\ e&f&g&0 \\ k&l&m&n \end{pmatrix}$.  
We divide the rest of the proof into two parts:\begin{enumerate} \item[1a)] ht$(e,f,g)=2$.  \item [1b)]  ht$(e,f,g)=3$ and ht$(k,l,m,n) \geq 3$. \end{enumerate} 

\noindent 1a) \underline{ht$(e,f,g)=2$:} We suppose $M=\begin{pmatrix} a&b&0&0\\ e&f&g&0 \\ k&l&m&n \end{pmatrix}$ with $n \neq 0$ and $g \neq0$ or else ht$(I_3(M))=1$. We can also assume that either $e =0$ or $f=0$ by a column operation and a linear change of variables. We write $M=\begin{pmatrix} a&b&0&0\\ 0&f&g&0 \\ k&l&m&n \end{pmatrix}$.
 
 $\bullet$ If ht$(a,b,f,g)=2$, then $M=\begin{pmatrix} a&b&0&0\\ 0&a&b&0 \\ k&l&m&n \end{pmatrix}$ with $n \neq 0$. We may assume that ht$(a,b,k,l,m,n) \geq 5$ or else we are back to \ref{form1}. 
 
  $\bullet$  If ht$(a,b,f,g)=3$, then the cases $f \in (a)$ or $g \in (b)$ are identical after a linear change of variable, which reduces our study to the following three case: $f \in (a)$, $f\in (b)$ or $g \in (a)$. When $g \in (a)$, we get ht$(I_3(M))=1$, which is impossible. If $ f \in (a)$, then $M = \begin{pmatrix} a&b&0&0\\ 0&a&g&0 \\ k&l&m&n \end{pmatrix}$ with $I_3(M)=(a^2n, agn, bgn, a^2m-agl+bgk)=(a,b)\cap (a^2,g) \cap (n, a^2m-agl+bgk)$ and $I_2(M)=(a^2, al-bk, ak, ag, bg, gk, gl, am, bm, an, bn, gn) = (a,b,g) \cap (a^2, k^2, ak, g,m,n, al-bk) \cap(a,b,k,l,n)$. Since ht$(a,b,g)=3$ and $a^2, bg \in I_2(M)$, then the only two cases to study are whenever $I_2(M) \subset (a,b)$ or $(a,g)$. The other cases give us either   ht$(a,b,g)<3$ or ht$(I_3(M))=1$. If $I_2(M) \subset (a,b)$, then $gk, gl, gn \in (a,b)$ which implies $k,l,n \in (a,b)$. Similarly when $I_2(M) \subset (a,g)$, we get $k,m,n \in (a,g)$. Both cases put us in \ref{form1}.
 Now, we suppose that $f \in (b)$, then $M = \begin{pmatrix} a&b&0&0\\ 0&b&g&0 \\ k&l&m&n \end{pmatrix}$. We repeat the same study as above. We have $I_3(M)=(bgk - agl + abm, abn, agn, bgn)=(a,b)\cap(a,g) \cap (b,g) \cap (n, abm-agl+bgk)$ and $I_2(M)=(ab, al-bk, bk, ag, am, gk, bg, bm, bm-gl,  an, bn, bn, gn)= (a, b, g)\cap (a,g,k,m,n) \cap (b,g,l,m,n)\cap (a,b,k,l,n)$. Again, since ht$(a,b,g)=3$ and $ag, bg \in I_2(M)$ then we study the cases when $I_2(M) \subset (a,b), (a,g)$ or $(b,g)$. We get $k,l,n \in (a,b)$, $k,m,n \in (a,g)$ or $l,m,n \in (b,g)$ respectively. This puts us one more time in \ref{form1}. 
  
 $\bullet$ If ht$(a,b,f,g)=4$, then $M = \begin{pmatrix} a&b&0&0\\ 0&f&g&0 \\ k&l&m&n \end{pmatrix}$ with $I_2(M) = (af, al-bk, fk, ag, am, gk, bg, bm, fm-gl, an, bn, fn, gn)= (a,b,f,g)\cap (a,b,k,n  ,fm-gl) \cap (a,g,k,m,n) \cap (al-bk, f,g,m,n)$ Since ht$(a,b,f,g)=4$ and $af,ag, bg \in I_2(M)$ then the only cases to study are when $I_2(M) \subset (a,b), (a,g)$ or $(b,g)$, we note that it would be impossible for $I_2(M)$ to be a subset of $(a,f), (b,g)$ nor $(b,f)$. If $I_2(M) \subset (a,b)$, then $fk, fm, fn, gl, gk, gn \in (a,b)$ which implies $k,l,m,n \in (a,b)$. This puts us in \ref{form1}. If $I_2(M) \subset (a,g)$, then $ fk, fm, fn, bk, bm,  bn, \in (a,g)$ and $k,m,n \in (a,g)$. Hence the only two possibilities for $M$ are $M = \begin{pmatrix} a&b&0&0\\ 0&f&g&0 \\ 0&l&a&g \end{pmatrix}$ or  $M = \begin{pmatrix} a&b&0&0\\ 0&f&g&0 \\ g&l&0&a \end{pmatrix}$, otherwise ht$(I_3(M))=1$. We can rewrite both matrices $M$, after row/column operations and a linear change of variables, as $M = \begin{pmatrix} b&a&0&0\\ f&a&g&0 \\ l&0&a&g \end{pmatrix}$ or $M = \begin{pmatrix} b&0&a&0\\ f&g&a&0 \\ l&0&g&a \end{pmatrix}$  respectively. Both matrices are reduced to \ref{bg}. Finally, the case when $I_2(M) \subset (f,g)$ is identical to $I_2(M) \subset (a,b)$ by a linear change of variables.   \\
 -We then study the case when $M = \begin{pmatrix} 0&0&a&b\\ e&f&g&0 \\ k&l&m&n \end{pmatrix}$. Since ht$(e,f,g)=2$, we assume $g=0$. We write $M = \begin{pmatrix} a&b&0&0\\ 0&0&e&f \\ k&l&m&n \end{pmatrix}$.\\
  $\bullet$  If ht$(a,b,e,f)=2$, then $M=\begin{pmatrix} a&b&0&0\\ 0&0&a&b \\ k&l&m&n \end{pmatrix}$ with ht$(a,b,k,l,m,n)\geq 5$ which is case \ref{abab1}. \\
  $\bullet$ If ht$(a,b,e,f)=3$, then we may assume $e\in (a)$, and $M=\begin{pmatrix} a&b&0&0\\ 0&0&a&f \\ k&l&m&n \end{pmatrix}$. Since $ab,af$ and $a^2 \in I_2(M)=(al-bk, a^2 , am, ak, ab, bm, al, af, an, fk, bf, bn, fl, an-fm)= (a,b,f) \cap (a,b,k,l,m) \cap (a,f,k,m,n) \cap (a^2, b, k, f, l, m, n)$ and ht$(a,b,f)=3$ then the only cases to study are whenever  $I_2(M) \subset (a,b)$ or $(a,f)$. If $I_2(M) \subset (a,b)$ or $(a,f)$, then $fm, fk, fl \in (a,b)$ or $bk, bm,  bn \in (a,f)$ respectively. Hence $k,l,m \in (a,b)$ or $k, m, n \in (a,f)$, and we get back to \ref{form1}. \\
  $\bullet$ If ht$(a,b,e,f)=4$, then we write $M=\begin{pmatrix} a&b&0&0\\ 0&0&e&f \\ k&l&m&n \end{pmatrix}$. Similarly to the above cases, we have $ae, af, be, bf \in I_2(M) =(al-bk, ae, am, ek, be, bm, el, af, an, fk, bf,  bn, fl, en-fm)= (a,b,e,f)\cap (al-bk, e, f, m, n) \cap (fm-en, a, b, k, l)$. We only study the case when $I_2(M ) \subset (a,b)$. We get $k,l,m,n \in (a,b)$. This puts us in case \ref{form1}. \\
 \noindent  1b) \underline{ht$(e,f,g)=3$ and ht$(k,l,m,n)\geq 3$:} --Suppose $M = \begin{pmatrix} a&b&0&0\\ e&f&g&0 \\ k&l&m&n \end{pmatrix}$ with $n \neq 0$ and $I_2(M)=(af-be, al-bk, el-fk,em-gk,fm-gl, ag, am, bg, bm,  an,  bn,en, fn, gn)$.
 We begin by simplifying the matrix $M$ by looking at $M_3$. We know that  ht$(I_2(M_3))=2$, so $I_2(M_3)$ should be included in a height two prime ideal which is either generated by two linear forms or by one linear and one quadric. Looking at the form of the matrix $M$, we study the cases when $I_2(M_3) \subset (e,f),(e,g), (el-fk, n)$ and $(em-gk,n)$. All other cases put us in these mentioned cases or back to 1a) . Suppose $I_2(M) \subset (e,f)$, then $n \in (e,f)$ and $gk,gl \in (e,f)$. We only study the case when $n=e$, since when $n=f$ is similar. Suppose $n=e$, then $k,l \in (e,f)$. We may assume that $l=0$ by a row/column operations, and the only possibility for $M$ is $M = \begin{pmatrix} a&b&0&0\\ e&f&g&0 \\ f&0&m&e \end{pmatrix}$. We write $M=\begin{pmatrix} 0&0&a&b\\ m&e&f&0 \\ g&0&e&f \end{pmatrix}$ after row/column operations, and we look at $I_2(M)$. We have $ag, bg \in I_2(M)$, which implies that $a,b \in (e,f)$. Hence we either get back to 1a) or we get ht$(I_2(M))=1$ after row/column operations. Suppose $I_2(M_3) \subset (e,g)$, then $fk, fm, fn \in (e,g)$. Suppose $n=e$. We may assume that $m=0$ and $k=g$. Hence $M=\begin{pmatrix} a&b&0&0\\ e&f&g&0 \\ g&l&0&e\end{pmatrix}$, which can be written as $M=\begin{pmatrix} b&0&a&0\\ l&e&g&0 \\ f&0&e&g \end{pmatrix}$ after row/column operations. Since $af, al \in I_2(M)$ then $a \in (e,g)$ which put us in \ref{bg}. When $n=g$, we get $m=e$ and $k=0$. Hence $M=\begin{pmatrix} b&a&0&0\\ f&e&g&0 \\ l&0&e&g \end{pmatrix}$ after row/column operations, and again $a \in (e,g)$. This puts us in \ref{bg}. When $I_2(M_3) \subset (el-fk,n)$, then either $M = \begin{pmatrix} a&b&0&0\\ e&f&n&0 \\ k&l&0&n \end{pmatrix}$ or $M = \begin{pmatrix} a&b&0&0\\ e&n&f&0 \\ k&0&l&n \end{pmatrix}$. In the former case, ht$(I_2(M))=1$ and we are done. In the latter case we get $n^2 \in I_2(M)$. So we study the cases when $I_2(M) \subset (a,n), (b,n)$ or $(el-fk,n)$. When $I_2(M) \subset (a,n)$ then $el-fk \in (a,n)$. We may assume $f=k=a$ but $bl \in (a,n)$. Hence either $b=n$ or $l \in (a,n)$. We get back to case 1a. Similarly when $I_2(M) \subset (b,n)$, we get back to case 1a.
  When $I_2(M) \subset (el-fk,n)$, we get $ af, bf \in I_2(M)$. Hence $a,b \in (n)$ and ht$(I_2(M))=1$. The case when $I_2(M) \subset (em-gk, n)$ is treated the same way as whenever $I_2(M) \subset (el-fk, n)$, so we omit the details.\\
 
 \noindent -- Now we suppose $M = \begin{pmatrix} 0&0&a&b\\ e&f&g&0 \\ k&l&m&n \end{pmatrix}$. We conduct the same proof as above. Since ht$(I_2(M_3))=2$ and ht$(e,f,g)=3$, we study the cases when $I_2(M_3) \subset (e,f), (e,g)$, $(em-gk,n)$ and $(el-fk, n)$. When $I_2(M_3) \subset (e,f)$, we get $gl, gk, gn \in (e,f)$ and hence $k,l,n \in (e,f)$. We may assume that $n \neq 0$ since otherwise we go back to case 1a) after row/column operations because $k,l \in (e,f)$. We assume $n=e$, the case when $n=f$ is similar and we will omit the details. When $n=e$, the two possibilities for the matrix $M$ would be $M = \begin{pmatrix} 0&0&a&b\\ e&f&g&0 \\ f&0&m&e \end{pmatrix}$ or $M = \begin{pmatrix} 0&0&a&b\\ e&f&g&0 \\ 0&f&m&e \end{pmatrix}$. These matrices can be written as $M = \begin{pmatrix} a&b&0&0\\ m&e&f&0 \\ g&0&e&f \end{pmatrix}$ and $M = \begin{pmatrix} a&b&0&0\\ m&e&f&0 \\ g&0&f&e \end{pmatrix}$ after row/column operation. Since $bg, bm \in (e,f)$ then $b \in (e,f)$, and this puts us in \ref{bg}. When $I_2(M_3 ) \subset (e,g)$,  then $k,m,n \in (e,g)$. Again $n \neq 0$ and we study the case when $n =g$ first. The possibilities for the matrix $M$ are either $M = \begin{pmatrix} 0&0&a&b\\ e&f&g&0 \\ 0&l&e&g \end{pmatrix}$ or $M = \begin{pmatrix} b&0&a&b\\ e&f&g&0 \\ 0&l&e&g \end{pmatrix}$. In the former case we get $(a,b) \subset (e,g)$ and hence we are either back to case 1a) after row/column operations or we get ht$(I_2(M))=1$. In the latter case, we also get $(a,b) \subset (e,g)$ and back to case 1a) after row/column operations.   When $n=e$, we get $M = \begin{pmatrix} 0&0&a&b\\ e&f&g&0 \\ g&l&0&e \end{pmatrix}$ and $(a,b) \subset (e,g)$. We get back to case 1a)  after row/column operations.\\
  When $I_2(M_3) \subset (el-fk,n)$, then either $M = \begin{pmatrix} 0&0&a&b\\ e&n&f&0 \\ k&0&l&n \end{pmatrix}$ or $M = \begin{pmatrix} 0&0&a&b\\ e&f&n&0 \\ k&l&0&n \end{pmatrix}$. In both cases, we have $n^2 \in I_2(M)$. So we study the cases when $I_2(M) \subset (a,n), (b,n)$ and $(el-fk,n)$. In both matrices, when $I_2(M) \in (a,n)$ or $(b,n)$, we get to case back 1a. When $I_2(M) \subset (el-fk,n)$, we get
  $a,b \in (n)$ and we get ht$(I_2(M))=1$. The case when $I_2(M_3) \subset (em-gk,n)$ is identical and treated the same way as $I_2(M_3) \subset (el-fk,n)$, so we omit the details.

\noindent  \underline {Case 2):} Suppose $M = \begin{pmatrix} a&b&c&0\\ e&f&g&0 \\ k&l&m&n \end{pmatrix}$. Again we have ht$(I_2(M_1))=$ht$(I_2(M_3))=2$. We may assume that ht$(a,b,c)=$ ht$(e,f,g)=3$ and ht$(k,l,m,n) \geq 3$, or else we are back to Case 1. Since ht$(I_2(M_3))=2$, then we study the cases when $I_2(M_3) \subset (e,f)$ and $(el-fk, n)$. We follow the same argument as before. If $I_2(M_3) \subset (e,f)$, then $k,l,n \in (e,f)$ and we study the case when $n=e$ only (the case $n=f$ is identical). Hence, $M = \begin{pmatrix} a&b&c&0\\ e&f&g&0 \\ f&0&m&e \end{pmatrix}$ written as $M = \begin{pmatrix} c&b&a&0\\ g&f&e&0 \\ m&0&f&e \end{pmatrix}$. We get $ag, am, bg, bm \in (e,f)$ and hence $a,b \in (e,f)$. This puts us in \ref{bg}.  When $I_2(M_3) \subset (el-fk,n)$, we write $M = \begin{pmatrix} a&b&c&0\\ e&f&n&0 \\ k&l&0&n \end{pmatrix}$. We have $n^2 \in I_2(M)$. So we investigate the cases whenever $I_2(M) \subset (a,n), (c,n), (af-be, n)$ and $(el-fk, n)$.  When $I_2(M) \in (a,n)$, $(a,c)$ or $(af-be, n)$, we get that $el-fk$ belongs each of these prime ideals respectively. All these cases put us back in Case 1. When $I_2(M) \in (el-fk,n)$, we get $c \in (n)$ since $ck, cl \in (n)$ and ht$(I_2(M))=1$.

\noindent  \underline {Case 3):} Suppose $M = \begin{pmatrix} a&b&c&0\\ e&f&0&c \\ k&l&m&n \end{pmatrix}$. We may also assume ht$(a,b,c)=$ht$(e,f,c)=3$ and ht$(k,l,m,n) \geq 3$, or else we get back to Case 1. We have ht$(I_2(M_3))=2$, and again we study the cases when $I_2(M_3) \subset (e,f), (e,c)$ and $(el-fk, c)$. Suppose $I_2(M_3) \subset  (e,f)$, then since $c^2 \in I_2(M)$ we get $c \in (e,f)$ which puts is back in Case 1. If $I_2(M_3) \subset (e,c)$, then $fk, fm, fn \in (e,c)$ and $k,m,n \in (e,c)$. Suppose $m=c$ then $m=0$ after row operation and a linear change of variables. If $n=c$ then ht$(I_3(M))=1$, so we assume $n =e$. We also assume that $k=c$ or else we get back to Case 1. Hence $M = \begin{pmatrix} a&b&c&0\\ e&f&0&c \\ c&l&0&e \end{pmatrix}$ and again we go back to Case 1 after row/column operation.  We now suppose $m=e$, then $n=0$ by row/column operations and $k=c$ or else we get back to Case 1. Hence  $M = \begin{pmatrix} a&b&c&0\\ e&f&0&c \\ c&l&e&0 \end{pmatrix}$. Since $al, af  \in (e,c)$ and ht$(e,f,c)=$ht$(c,e,l)=3$, then $a \in (e,c)$ which puts us in any of the two forms of \ref{bg}.  
Finally suppose $I_2(M_3) \subset (el-fk, m)$, then since $c^2 \in I_2(M)$ we get $c=m$. We only study the case when $I_2(M) \subset (el-fk, m)$ since otherwise we get back to Case 1. We have $bn \in I_2(M)$ and ht$(a,b,c)=3$, so $n=m$ and hence $n=0$. So ht$(I_3(M))=1$ whenever $M=\begin{pmatrix} a&b&m&0\\ e&f&0&m \\ k&l&m&0 \end{pmatrix}$.\\

 \noindent  \underline {Case 4:)} Suppose $M = \begin{pmatrix} a&b&c&0\\ e&0&b&c \\ k&l&m&n \end{pmatrix}$. In this case, we do not need to study the cases when $I_2(M_3) \subset (e,b), (e,c), (b,c), (bn-mc,l) $ since $b^2, c^2 \subset I_2(M)=$. Hence we study the case when $I_2(M) \subset (b,c)$. We get $l,m,n \in (b,c)$ and we are in \ref{bg} after row/column operations. 
 
  \end{proof}
  
  \begin{remark}\label{4vars} We record the matrices of \ref{lemmaI_3(M)=2}{\color{blue}-}\ref{form1} as details are needed to complete the proof of Proposition \ref{mult3}:
  \begin {enumerate}
  \item $M = \begin{pmatrix} a&b&0&0\\ 0&a&b&0 \\ k&l&m&n \end{pmatrix}$ or $M = \begin{pmatrix} a&b&0&0\\ 0&0&a&b \\ k&l&m&n \end{pmatrix}$ with ${ \small \left\{ \begin{matrix} \mbox{ht}(a,b,n)=3, \hspace {0.2 cm} \mbox {and} \hspace {0.2 cm}  3 \leq \mbox{ht}(a,b,k,l,m,n) \leq 4 \hspace {0.2 cm} \mbox {or}  \\ n=a, \mbox{ht}(k,n)=2, \hspace {0.2 cm} \mbox {and} \hspace {0.2 cm}  3 \leq \mbox{ht}(a,b,k,l,m,n) \leq 4  \hspace {0.2 cm} \mbox {or} \\ n=b, \mbox{ht}(m,n)=2, \hspace {0.2 cm} \mbox {and} \hspace {0.2 cm}  3 \leq \mbox{ht}(a,b,k,l,m,n) \leq 4.  \end{matrix} \right.}$
  \item $M = \begin{pmatrix} a&b&0&0\\ 0&a&b&0 \\ 0&0&a&b \end{pmatrix}$ with ht$(a,b)=2$.
   \item $M = \begin{pmatrix} d&m&0&0\\ c&b&a&0 \\ 0&0&b&a \end{pmatrix}$ or  $M = \begin{pmatrix} d&m&0&0\\ c&0&b&a \\ 0&b&a&0 \end{pmatrix}$ with $m \in (a,b)$ and ht$(a,b,c,d)=4$
  \item $M = \begin{pmatrix} a&b&0&0\\ c&b&a&0 \\ d&0&b&a \end{pmatrix}$ or  $M = \begin{pmatrix} a&0&b&0\\ c&b&a&0 \\ d&0&b&a \end{pmatrix}$  with ht$(a,b,c)=$ht$(a,b,d)=3$
  \item $M = \begin{pmatrix} d&a&m&0\\ c&b&a&0 \\ 0&0&a&b \end{pmatrix}$ or $M = \begin{pmatrix} d&b&m&0\\ c&0&b&a \\ 0&a&b&0 \end{pmatrix}$ with $m \in (a,b)$ and ht$(a,b,c,d)=4$.
  \end{enumerate}
  \end {remark}

\begin{proposition} \label{mult3} Let $J$ be an $(x,y,z,w)$-primary with $e(S/J)=3$. Then one of the following holds
\begin{itemize} 
\item[I-] If $J$ is degenerate, then 
 \begin{enumerate}
  \item \label{3lf} \small $J=(x,y,z,w^3)$.
\item \label{2lf1}$J = (x,y)+(z,w)^2$.
 \item   \label{2lf2} $J=(x,y)+ (z^2, wz, w^3, w^2+az)$  with ht$(x,y,z,w,a)=5.$
 \item   \label{3lf1} $J=(x, y)+ (z^3, z^2w, zw^2, w^3, az+bw ) $ with ht$(x,y,z,a,b)=5.$
\item  \label{lf1} $J=(x)+ (y,z,w)^2+(by+cz+dw)$ with ht$(x,y,z,w,b,c,d) \geq 6$.
\item \label{lf2} $J=(x)+ (y,z)^2+( wy, wz, w^2+by+cz, ey+fz)$ with ht$(x,y,z,w,e,f)=6$
\item \label {lf3}$J=(x)+ y(y, z,w)+ (z,w)^3+(  ay+bz+cw, dy+z^2 )$ with { {ht$(x,y,z,w,b,c)=\\
 $ht$(x,y,z,w,c,d)=6$.}}
\item \label {lf4}$J=(x)+y(y, z, w)+(z,w)^3+  (ay+bz+cw, cy+z^2,by-zw)$ with \mbox {ht$(x,y,z,w,b,c)=6$}.
\item \label {lf5}$ J=(x)+y(y, z, w)+ (z,w)^3+(a y+bz+cw, dy+zw )$ with ht$(x,y,z,w,d,c,d)=7.$
\item \label {lf6}$J=(x)+y(y, z, w)+ (z,w)^3 +(a y+bz+cw, by+zw, cy-z^2 )$ with \mbox{ht$(x,y,z,w,b,c)=6$.}
\item \label {lf7}$J=(x)+y(y, z, w)+ (z,w)^3+  (a y+bz+cw, cy+zw, by-w^2 )$ with \mbox{ ht$(x,y,z,w,b,c)=6$.}
 \item \label {height2} All quadrics in $J$ generate an ideal of height at most three.
 \item \label {varss} All quadrics in $J$ can be expressed in terms of at most  $8$ variables.
 
\end{enumerate}
\vspace{0,5cm}
\item[II-] If $J$ is non-degenerate, then 
\begin{enumerate}
\item \label{Case1a} $J=(x,y,z,w)^2+(ax+by+cz+dw, ex+fy+gz+hw)$ with \small {ht$\left(I_2 \begin{pmatrix}a&b&c&d\\ e&f&g&h \end{pmatrix}\right) \geq 2$.}
\item \label{Case1b} $J=(x,y,z)^2+w(x,y,z)+(w^2+ ax+by+cz, ex+fy+gz, kx+ly+mz)$ with \mbox{ht$\left(I_2\begin{pmatrix} e&f&g\\k&l&m\end{pmatrix}\right) \geq 2$.} 
\item  \label{Case1c} $J=(x,y,z)^2+w(x,y,z,w)+( ax+by+cz, ex+fy, kx+fz, ky-ez)$ with ht$(e,f,k)=3$.
\item  \label{wz}  $J= (x,y)^2+(x,y)(z,w)+( ax+by+cz+dw, ex+fy+z^2+\alpha w^2, gx+hy+zw)$ with ht$(x,y,z,w, c,d)=6$, ht$(c,g,h) \geq 2$, ht$(d,g,h) \geq 2$, ht$(e,f,g,h) \geq 2$ and ht$\left(I_2  \begin{pmatrix}e&f\\ g&h \end{pmatrix}\right) =1$, ($\alpha \in k$).
\item \label {e=h} $J=(x,y)^2+ ( x,y)(z,w)+( ax+by+cz+dw, ex+z^2+ w^2, ey+zw, c^2y+cdx+d^2y)$ with ht$(c,d,e)=3$.
\item \label {e=h,0} $J=(x,y)^2+ ( x,y)(z,w)+(ax+by+cz+dw, ex+z^2, ey+zw, cx+dy)$ with ht$(c,d,e)=3$.
\item  \label {g=h=a} $J=(x,y)^2+ ( x,y)(z,w)+(  ax+by+cz+dw, ex+fy+ z^2+\alpha w^2,  cx+zw, dx-z^2 )$ with \mbox{ ht$(c,d)=$ht$(e,f)= 2$.}
\item \label {z^2, w^2=0} $J=(x,y)^2+ ( x,y)(z,w)+(ax+by+cz+dw, ex+fy, gx+hy+z^2)$ with ht$(c,d)=$ht$(e,f)=2$.
\item \label {z^2} $J=(x,y)^2+ ( x,y)(z,w)+ (ax+by+cz+dw, ex+fy+w^2, gx+hy+z^2)$ with ht$(x,y,z,w, c,d)=6$, ht$(c,e,f) \geq 2$, ht$(d,g,h) \geq 2$ and ht$\small \left( I_2 \begin{pmatrix}e&f\\ g&h \end{pmatrix} \right)=1$.
\item \label{z^2,e=a} $J=(x,y)^2+ ( x,y)(z,w)+ (ax+by+cz+dw, cx+w^2, gx+hy+z^2, dx-zw)$ with ht$(x,y,z,w,c,d)=6$ and ht$(d,g,h) \geq 2$.
 \item \label {height} All quadrics in $J$ generate an ideal of height at most three.
 \item \label {vars} All quadrics in $J$ can be expressed in terms of at most  $8$ variables.
 \end{enumerate}
 \end{itemize}
 \end{proposition}
 
 \begin{proof} \normalsize Since $J \subset J:\mathfrak p $ then $e(S/J) > e(S/(J:\mathfrak p))$. For that $e(S/(J:\mathfrak p))=1$ or $2$.\\
  If $e(S/(J: \mathfrak p))=1$, then $(J:\mathfrak p)=(x,y,z,w)$ and $(x,y,z,w)^2 \subset J$. If $J$ contains three linear forms then $ (x,y,z,w^2) \subset J$ and $e(S/(x,y,z,w^2))= 2 \geq e(S/J)$, which is a contradiction. If $J$ contains two linear forms then $J=(x,y)+(z,w)^2$, which puts us in case {\color{blue}I-}\ref{2lf1}. If $J$ contains one linear form then $(x)+(y,z,w)^2 \subset J$, and $J$ has at least one more element of the form $q=by+cz+dw$.  The ideal $(y,z,w)^2+(x, by+cz+dw)$ is unmixed whenever ht$(x,y,z,w,b,c,d) \geq 6$ by Lemma \ref{lemma4}, which is case {\color{blue}I-}\ref{lf1}. If ht$(x,y,z,w,b,c,d)=5 $, then we can modify $c,d$ modulo $(x,y,z,w)$ to assume that $c$ and $d$ are multiple of $b$. Hence $q=bl$ for some linear form $l$, and $J$ contains another linear form, which is a contradiction. If $J$ does not contain any linear form, then since $e(S/(x,y,z,w)^2)=5$, $J$ must contain at least two generators of the form $(ax+by+cz+dw,ex+fy+gz+hw)$. If ht$\left(I_2 \begin{pmatrix}a&b&c&d\\ e&f&g&h \end{pmatrix}\right) \geq 2$, then $(x,y,z,w)^2+(ax+by+cz+dw, ex+fy+gz+hw)$ is unmixed of multiplicity $3$ and hence equal to $J$ by Lemma \ref{lemma5}, which puts us in case {\color{blue}II-}\ref{Case1a}. When ht$\left(I_2 \begin{pmatrix}a&b&c&d\\ e&f&g&h \end{pmatrix}\right)= 1$ modulo $(x,y,z,w)$, then $M$ has one of the form of Lemma \ref{lemmaI_2}. In the first two cases, $J$ contains a linear form. In the last case, $(x,y,z,w)^2+(ax+by, ex+bz)\subset J$ and $b(ey-az) \in J$. Since $b \notin J$ then $ey-az \in J$. Hence by lemma  \ref{lemma6}, $J=(x,y,z,w)^2+(ax+by, ex+bz, ey-az)$, which puts us in {\color{blue}II-}\ref{vars}.  \\
  
  Suppose $e(S/(J: \mathfrak p)=2$ then there are $6$ cases to consider according to Proposition \ref{multiplicity2}. The case when all quadrics $J: \mathfrak p$ are written in terms of $8$ variables put us in {\color{blue}I-}\ref{varss} or {\color{blue}II-}\ref{vars} since $J \subset J: \mathfrak p$. \\
  
  \underline{Case I. $J: \mathfrak p  = (x,y,z,w^2)$}\\
  In this case, we have $(x,y,z)^2+(wx,wy,wz,w^3) \subset J \subset (x,y,z,w^2)$. If $J$ contains three independent linear forms then $J = (x,y,z,w^3)$, which is case {\color{blue}I-}\ref{3lf}. If $J$ contains two linear forms say $x, y$ then $(x,y,z^2, wz, w^3) \subset J$. Since this ideal on the left has multiplicity $4$, then we need one additional generator of the form $w^2+az$. By Lemma \ref{lemma7}, the ideal $(x,y,z^2, wz, w^3, w^2+az)$ is unmixed of multiplicity $3$, and hence equal to $J$. If ht$(x,y,z,w,a)=5$ we are in case {\color{blue}I-}\ref{2lf2}, otherwise when $a \in (x,y,z,w)$ we are in case {\color{blue}I-}\ref{2lf1}. If $J$ contains a linear form $x$, then $(x)+(y,z)^2+(wy, wz, w^3) \subset J$. Since the multiplity of the ideal on the left is $5$, then either $J$ has a generator of the form $w^2+by+cz+dw$ with $d\neq 0$, or two generators of the form $w^2+by+cz$ and $ey+fz$. In the first case, we have $dw^2 \in J$ since $d \notin (x,y,z,w)$ then $w^2 \in J$. We get the ideal  $(x,by+cz+dw)+(w,y,z)^2$, which by lemma \ref{lemma8}, is unmixed of height four and multiplicity 3 when ht$(x,y,z,w,b,c,d)\geq 6$. This puts us in case {\color{blue}I-}\ref {lf1}, or back to {\color{blue}I-}\ref{2lf1} when ht$(x,y,z,w,b,c,d)=5$.  In the second case, one of the quadrics must contain $w^2$, or else $J$ contains a second linear form. The ideal $(x)+(y,z)^2+( wy, wz, w^2+by+cz, ey+fz)$ with ht$(x,y,z,w,e,f)=6$ is umixed of multiplicity $3$ by Lemma \ref {lemma9}, and hence is equal to $J$. This puts us in case {\color{blue}I-}\ref{lf2}. If ht$(x,y,z,w,e,f)=5$, then $J$ contains a second linear form $y$ or $z$.  Suppose $J$ doesn't contain a linear form. Since the multiplicity of the ideal  $(x,y,z)^2+(wx,wy,wz,w^3)$ is $6$, then $J$ must contain either two generators of the form $w^2+ax+by+cz+dw$ and $ex+fy+gz+hw$ with $d \neq 0$ or $h \neq 0$, or three linearly independent quadrics of the from  $w^2+ ax+by+cz$, $ex+fy+gz$ and $kx+ly+mz$. In the first case $w^2 \in J$, and by lemma \ref{lemma5}, we are in case {\color{blue}II-}\ref{Case1a} when ht$\left(I_2 \begin{pmatrix}a&b&c&d\\ e&f&g&h \end{pmatrix} \right) \geq 2$. Otherwise, either $J$ contains a linear form or $J$ is in {\color{blue}II-}\ref{vars}. In the second case, by Lemma  \ref{lemma10}, the ideal $(x,y,z)^2+(wx,wy,wz,w^2, ax+by+cz, ex+fy+gz, kx+ly+mz)$ is unmixed of multiplicity $3$ when ht$\left(I_2\begin{pmatrix} e&f&g\\k&l&m\end{pmatrix}\right) \geq 2$ modulo $(x,y,z,w)$,  and hence is equal to $J$. This puts us in case {\color{blue}II-}\ref{Case1b}. If ht$\left(I_2\begin{pmatrix} e&f&g\\k&l&m\end{pmatrix}\right) =1$, then by \cite[Lemma 4.1]{HMMS13}, $M= \begin{pmatrix} e&f&0\\k&0&f\end{pmatrix}$ with ht$(e,f,k)=3$ or else $J$ contains a linear form. In that case, $(x,y,z)^2+(wx,wy,wz,w^2+ ax+by+cz, ex+fy, kx+fz)\in J$ and $f(ky-ez) \in J$. Since $f \notin (x,y,z,w)$, then $ky-ez \in J$. By Lemma \ref{lemma11}, $(x,y,z)^2+(wx,wy,wz,w^2+ ax+by+cz, ex+fy, kx+fz, ky-ez)$ is unmixed of multiplicity $3$, and hence equal to J. This puts us in case {\color{blue}II-}\ref{Case1c}.\\
  
   \underline{Case II. $J: {\mathfrak p} = (x,y,z^2, zw, w^2, az+bw )$, ht$(x,y,z,w,a,b) =6$.} \\
 In this case, $(x^2, xy, xz, xw, y^2, yz, yw, z^3, z^2w, zw^2,  w^3, z(az+bw),  w(az+bw)) \subset J \subset J: \mathfrak{p}$. 
   If $J$ contains the two linear forms $x$ and $y$ then $(x, y, z^3, z^2w, zw^2, w^3, z(az+bw),  w(za+bw)) \subset J$. There is at least one generator of the form $\alpha z^2+\beta zw + \gamma w^2 + \delta (az+bw)$, which we may assume to be quadric, otherwise $J$ is in {\color{blue} I-}\ref{height2}. If $\delta=0$, then $(x,y,z,w)^2_{\mathfrak{p}} \subset J_{\mathfrak{p}}$ and we get $(x,y)+(z,w)^2 \subset J$ since $J$ is $\mathfrak{p}$-primary. Both ideals have the same height and multiplicity, and hence are equal. So $J=(x,y)+(z,w)^2 $, which is impossible since $J:\mathfrak{p}= (x,y,z^2, zw, w^2, az+bw )$. So we assume $ \delta=1$. Let $a' = a+\alpha z+ \beta w$ and $b'=  b+ \gamma w$ we get $(x, y, z^3, z^2w, zw^2, w^3, az+bw ) \subset J$. The ideal on the left is unmixed by Lemma  \ref{lemma12} whenever ht$(x,y,z,w,a,b)=6$. Since both ideals have the same height and multiplicity then they are equal. Our primary ideal falls into  {\color{blue} I-}\ref{3lf1}. If ht$(x,y,z,w,a,b) =5$ then $J$ contains the linear form $z+w$, and $J$ will have the form of {\color{blue} I-}\ref{3lf} after a linear change of variables. 
   
  Suppose $x \in J $ and $y \notin J$, then $(x,y^2, yz, yw, z^3, z^2w, zw^2,  w^3, z(az+bw),  w(az+bw)) \subset J$. Since the multiplicity of the ideal on the left is $5$, then there exists at least $2$ additional generators of the form $\alpha y+\beta z^2+\gamma wz+\theta w^2+\delta (az+bw)$, $cy+dz^2+ewz+fw^2$  with $\deg c =\deg \alpha =1$. The remaining coefficients are in $k$ otherwise we will be in case  {\color{blue} I-}\ref{height2}. 
   Suppose $\delta =0$ for all such quadrics, then since there are at most three of them in $J$, we get  $gy \in J$ for some $g \notin (x,y,z,w)$ which implies $y \in J$ impossible. Hence, we may assume that $\delta =1$. After the following linear change of variables $(a'= a+\beta z +\gamma w $ and $b'= b+ \theta w)$, we can consider $\beta=\gamma=\theta=0$. After a further linear change of variables,  the second quadric has either one of the two following forms $cy+ z^2$ or $cy+zw$  since any polynomial with two variables over an algebraically closed field splits. 
Suppose the second quadric is $cy+ z^2$ then $(x,y^2, yz, yw, z^3, z^2w, zw^2,  w^3, \alpha y+az+bw, cy+z^2 ) \subset J$. Note that ht$(x,y,z,w,a,b)=6$. If ht$(x,y,z,w,b,c)=6$, then by Lemma \ref{lemma13}, we get $J=(x,y^2, yz, yw, z^3, z^2w, zw^2,  w^3, \alpha y+az+bw, cy+z^2 )$ which is case {\color{blue} II-}\ref {lf3}.  When  ht$(x,y,z,w,b,c)=5$, then we may assume either $c=0$ or $b=c$. If $c=0$, then $(y,z,w)^2+(x,\alpha y + az+bw) \subset J$.  Since the ideal on the left is unmixed by lemma \ref{lemma4} of multiplicity $3$ then it is equal to $J$, which is case {\color{blue} I-}\ref{lf1}. If $b=c$ then $aby -bzw \in J $ with $b \notin \mathfrak{p} $, then $ay-zw \in J$. By lemma \ref{lemma15} $J=(x,y^2, yz, yw, z^3, z^2w, zw^2,  w^3, \alpha y+az+bw, by+z^2, ay-zw)$, which is {\color{blue} I-}\ref{lf4}.\\
Similarly, if the second quadric is $cy+zw$, then $(x,y^2, yz, yw, z^3, z^2w, zw^2,  w^3, \alpha y+az+bw, cy+zw ) \subset J$. If ht$(x,y,z,w,a,b,c)=7$, then by lemma \ref{lemma14}, both ideals are equal and this falls in case {\color{blue} I-}\ref{lf5}. Otherwise, we may assume that $c=a$, $c=b$ or $c=0$. If $c=a$ or $c=b$ then by lemma \ref{lemma15}, $J$ has the form of {\color{blue} I-}\ref{lf6} or {\color{blue} I-}\ref{lf7} . If $c=0$, then $z^2$ and $w^2 \in J$, and $J$ has the form of {\color{blue} I-}\ref{lf1}  by lemma  \ref{lemma4}. \\

If $x,y \notin J$ then $(x^2, xy, xz, xw, y^2, yz, yw, z^3, z^2w, zw^2,  w^3, z(az+bw),  w(az+bw)) \subset J$. The multiplicity of the ideal on the left is $6$, so  there exists at least three generators of the form $cx + d y+a_1 z^2+a_2 wz+a_3 w^2+\delta (az+bw)$, $ex + fy+\alpha_1z^2+\alpha_2wz+b_3w^2$ and $gx+ hy+c_1z^2+c_2wz+c_3w^2$. We may assume that $c, d, e, f, g, h  \in S_1$ and the rest of the coefficients are in $k$, or else we get to case {\color{blue}II-\ref{height}}. We may also assume that $\delta \neq 0$ since otherwise when $\delta = 0$ for all such quadrics, then $(x, y, z^2, zw, w^2)_{\mathfrak {p}} \subset J_{\mathfrak{p}}$. We get $(x, y, z^2, zw, w^2) \subset J$ since $J$ is $\mathfrak{p}$-primary,  which is a contradiction. Take $\delta = 1$.  After the following linear change of variables: $a'=a+a_1z+a_2w$ and $b'= b+a_3w$, we may assume that $a_1=a_2=a_3=0$. After a further linear change of variables, we may assume that the last quadric has the form of either $gx+ hy+wz$ or $gx+ hy+z^2$. Suppose the last quadric is $gx+ hy+wz$. Hence $(x^2, xy, xz, xw, y^2, yz, yw, z^3, z^2w, zw^2, w^3, cx+dy+az+bw, ex+fy+\alpha_1z^2+\alpha_2w^2, gx+hy+zw ) \subset J$. Further, we may assume that $\alpha_1$ or $\alpha_2 \neq 0$, or else the multiplicity of the ideal on the left is $4$. We consider $\alpha_1=1$.
If ht$(a,g,h) \geq 2$, ht$(b,g,h) \geq 2$, ht$(e,f,g,h) \geq 2$ and ht$\left(I_2  \begin{pmatrix}e&f\\ g&h \end{pmatrix}\right) =1$ then, by Lemma \ref{lemma16}, $J= (x^2, xy, xz, xw, y^2, yz, yw, cx + dy+az+bw, ex+fy +  z^2+\alpha w^2, gx+hy+zw)$ with $\alpha \in k$.  This is case {\color{blue} II-}\ref{wz}.
If $I_2  \begin{pmatrix}e&f\\ g&h \end{pmatrix} =0$, then either $J$ is degenerate or we get back to case {\color{blue} II-}\ref{Case1a}.  If ht$(e,f,g,h)=1$, then the only case to consider is whenever $f=g=0$ and $h=e$, or else $J$ is degenerate.  By Lemma \ref{lemma17}, if ht$(a,b,e)=3$ then $J=(x^2, xy, xz, xw, y^2, yz, yw, cx+dy+az+bw, ex+z^2+\alpha_2 w^2, ey+zw,  \alpha_2 a^2y+abx+b^2y)$ or $J=(x^2, xy, xz, xw, y^2, yz, yw, cx+dy+az+bw, ex+z^2+\alpha_2 w^2, ey+zw,  \alpha_2 a^2y+abx+b^2y)$ when $\alpha =0$. This puts us in case  {\color{blue} II-}\ref{e=h} or  {\color{blue} II-}\ref{e=h,0} when $\alpha=0$. Both cases ht$(a,g,h)=1$ or ht$(b,g,h)=1$ are equivalent by a linear change of variables. Suppose ht$(a,g,h)=1$ and $g=h=a$.  By a linear change of variables, we may assume $g=a$ and $h=0$. Hence, by Lemma \ref{lemma18}, $J=(x^2, xy, xz, xw, y^2, yz, yw, cx+dy+az+bw, ex+fy+\alpha_1 z^2+ \alpha_2 w^2,  ax+zw, bx-z^2 )$ with ht$(e,f)= 2$ or else $J$ is degenerate . This puts us in case  {\color{blue}II-\ref {g=h=a}}.\\ Suppose the second quadric is $gx+hy+z^2$, then $(x^2, xy, xz, xw, y^2, yz, yw, z^3, z^2w, zw^2, w^3, cx+dy+az+bw, ex+fy+\beta zw+ \alpha w^2, gx+hy+z^2 ) \subset J$. We may assume $\beta=0$ or else we get back to the first case by a linear change of variables. If $\alpha=0$ then $J=(x^2, xy, xz, xw, y^2, yz, yw, cx+dy+az+bw, ex+fy, gx+hy+z^2)$ when ht$(e,f)=2$ by Lemma \ref{lemma19}. This is case {\color{blue} II-}\ref{z^2, w^2=0}.  Suppose $\alpha \neq0$. If ht$(a,e,f) \geq 2$, ht$(b,g,h) \geq 2$ and ht$\left( I_2 \begin{pmatrix}e&f\\ g&h \end{pmatrix} \right)$,  then $J=(x^2, xy, xz, xw, y^2, yz, yw, cx+dy+az+bw, ex+fy+w^2, gx+hy+z^2)$ by Lemma \ref{lemma20}. This puts us in case {\color{blue}II-}\ref{z^2}. If $ I_2 \begin{pmatrix}e&f\\ g&h \end{pmatrix}=0$, then either $J$ is degenerate or we we get back to case {\color{blue} II-}\ref{Case1a}.  If ht$(a,e,f)=1$ (or equivalently ht$(b,g,h)=1$ by a change of variables) , then  by Lemma \ref{lemma21},  $J=(x^2, xy, xz, xw, y^2, yz, yw, cx+dy+az+bw, ax+w^2, gx+hy+z^2, bx-zw)$ with ht$(b,g,h) \geq 2$. This is case {\color{blue}II-}\ref{z^2,e=a}. If ht$(a,e,f)=$ht$(b,g,h)=1$, then we get back to case {\color{blue} II-}\ref{vars}.

\vskip 0,3cm

   \underline{Case III. $J:P =(y,z,w)^2+ (x,by+cz+dw, ey+fz+gw)$.}\\
   In this case, we have $(x,y,z,w)(y,z,w)^2+ (x,y,z,w)(x,by+cz+dw, ey+fz+gw)$ with \\
   ht$\left(x,y,z,w,I_2\begin{pmatrix}b&c&d\\ e&f&g \end{pmatrix}\right) \geq 6$. If $J$ contains $x$ then $(x)+(y, z, w)^3+(y, z, w)(by+cz+dw, ey+fz+gw) \subset J$ and the study of this case was done in case III of \cite [Proposition 4.4]{HMMS13}. This puts us in {\color{blue}I-\ref{height2}}. Suppose $x \notin J$ then the above ideal on the left has multiplicity $6$. The Hilbert function of the ideal on the left localized at $\mathfrak {p}$ is $(1, 4,1)$. So $J$ contains at least two quadrics which, after a linear change of variables, we may consider to be $ax+by+cz+dw, hx+ey+fz+gw$. Hence $(y,z,w)^3+ (x^2,xy,xz,xw)+(ax+by+cz+dw, hx+ey+fz+gw) \subset J$ and the multiplicity of the ideal on the left is $4$. If $J$ contains one more quadric of the above form then the multiplicity will be $2$. On the other hand, we have $(y,z,w)(by+cz+dw, ey+fz+gw) \subset J$ which leads to $(x,y,z,w)^2_{\mathfrak{p}} \subset J_{\mathfrak{p}}$. Since $J$ is $\mathfrak{p}$-primary, then $(x,y,z,w)^2+ (ax+by+cz+dw, hx+ ey+fz+gw) \subset J$. If ht$\left(I_2\begin{pmatrix}a&b&c&d\\ h&e&f&g \end{pmatrix}\right) \geq 2$ then by lemma \ref{lemma5} both ideals are equal which puts us in case {\color{blue}II-\ref{Case1a}}. If  ht$\left(I_2\begin{pmatrix}a&b&c&d\\ h&e&f&g \end{pmatrix}\right)=1$, then by lemma \ref{lemmaI_2} either $J$ is degenerate or  the quadrics in $J$ are expressed by at most $8$ variables which puts us in case {\color{blue}II-\ref{vars}}.\\
   
   \underline{Case IV. $J:P = (x, y, z,w)^2+ (ax+by+cz+dw, ex+fy+gz+hw, kx+ly+mz+nw)$.}\\
  We have $ (x, y, z,w)^3+ (x,y,z,w)(ax+by+cz+dw, ex+fy+gz+hw, kx+ly+mz+nw) \subset J$ with ht$(x,y,z,w,I_3(M))\geq 6$, where $M=\begin{pmatrix} a&b&c&d\\ e&f&g&h\\k&l&m&n \end{pmatrix}$. Since the multiplicity of the ideal on the left is $6$, then there exist at least three additional quadrics, which after a linear change of variables, we can take to be of the form $ax+by+cz+dw+q, ex+fy+gz+hw+q'$ and $kx+ly+mz+nw+q"$, $q,q',q" \in (x,y,z,w)^2$. After relabeling $a..h, k..n$, we can assume $q,q',q"=0$, without changing ht$(x,y,z,w,I_3(M))$. Hence  $ (x, y, z,w)^3+ (ax+by+cz+dw, ex+fy+gz+hw, kx+ly+mz+nw) \subset J$.  If ht$(I_3(N)) \geq 3$ where  $$\tiny N=\left(
\begin{array}{*{14}c}
 0& 0 & 0 & 0 & a & b & c & d & e & f & g & h \\
 -a& -b & -c & -d & 0 & 0 & 0 & 0 &  k & l &m &n \\
-e& -f& -g&-h&-k&-l&-m&-n& 0&0&0&0\\
\end{array}
\right),$$ then the ideal on the left is unmixed of multiplicity three and hence equal to $J$ by lemma \ref{lemma25}. 
  Suppose ht$(I_3((N)) \leq 2$. We have  $2 \leq$ ht$(I_3(M)) \leq $ ht$(I_2(M))$ mod $(x,y,z,w)$.  If ht$(I_2(M)) \geq 3$, then there exists at least three minors $\Delta_1,\Delta_2$ and $\Delta_3$ such that  $(\Delta_1,\Delta_2,\Delta_3)$ is of height three. We may assume $\Delta_1=(af-be)$. We get $a\Delta_1$, $b\Delta_1$, $e\Delta_1$ and $f\Delta_1\in I_3(N)$ and hence $\Delta_1^2 \in I_3(N)$, similarly for $\Delta_2,\Delta_3$. Hence, we obtain ht$(I_3(N))\geq 3$. So the only remaining case to study is whenever ht$(I_3(M))=$ ht$(I_2(M))=2$.
  By lemma \ref {lemmaI_3(M)=2}, $M$ takes four different forms. If $M$ is of the form \ref{lemmaI_3(M)=2}{\color{blue}-}\ref{abab} then $a(xz-y^2)=z(ax+by) -y(ay+bz) \in J$, and hence $J = (x,y,z,w)^3+(ax+by, ay+bz, xz-y^2, kx+ly+mz+nw)$ by lemma \ref{lemma23} which is case {\color{blue} II-}\ref{height}.  If $M$ is of the form \ref{lemmaI_3(M)=2}{\color{blue}-}\ref{abab1} then $b(zy-xw)=z(ax+by) -x(az+bw) \in J$, and hence $J = (x,y,z,w)^3+(ax+by, ay+bz, zy-xw, kx+ly+mz+nw)$ by lemma \ref{lemma23} which is case {\color{blue} II-}\ref{height}. If $M$ has the form of \ref{lemmaI_3(M)=2}{\color{blue}-}\ref{form1}, then by Remark \ref{4vars} and Lemmas \ref{lemma37}, \ref{lemma38}, \ref{lemma39}, we are in case {\color{blue} II-}\ref{vars}. If $M$ has the form of \ref{lemmaI_3(M)=2}{\color{blue}-}\ref{bg}, then we only study one case since all other cases are treated in the same way.  We study the case when $\alpha_2=\beta_1=\beta_2=0$ in the ideal $(x,y,z,w)^3+(ax+(\alpha_1f+\alpha_2g)y+(\beta_1f+\beta_2g)z, ex+fy+gz, kx+fz+gw)$. We get  $f(z^2a-ywa+ywe-yzk)= (we-kz)(ax+fy)-wa(ex+fy+gz)+az(kx+fz+gw)\in J$ and hence $J = (x,y,z,w)^3+(ax+fy, ex+fy+gz,kx+fz+gw, z^2a-ywa+ywe-yzk)$ by lemma \ref{lemma24}. Similarly, in the second form of the matrix we also study the case when $\alpha_2=\beta_1=\beta_2=0$ in the ideal $(x,y,z,w)^3+(ax+(\alpha_1f+\alpha_2g)y+(\beta_1f+\beta_2g)z, ex+fy+gz, kx+gz+fw)$. We get  $J=(x,y,z,w)^3+(ax+gy+fz,ex+fy+gz , kx+fz+gw, ya- wa - ye + yk)$ by lemma \ref{lemma24}. All other cases for $\alpha_i$ and $\beta_j$ are identical and we will omit the details. 
  These cases put us in {\color{blue} II-}\ref{height}.\\

  \underline{Case V. $J:P=(x,y,z,w)^2+(ax+by, ex+fy+nz, kx+ly+nw, (af-be)w-(al-kb)z)$.}\\
 We have $(x, y, z,w)^3+ (x,y,z,w)(ax+by, ex+fy+nz, kx+ly+nw, (af-be)w-(al-kb)z) \subset J$. Since the multiplicity of the ideal on the left is $6$, then there exit at least three additional genrators which we may assume to be of the following $g_1=ax+by+q_1, g_2=ex+fy+nz+q_2, g_3=  kx+ly+nw +q_3$ and $g_4= (af-be)w-(al-kb)z+q_4$ with $q_1, q_2, q_3 \in (z,w)^2$ and $q_4 \in (x,y)^2$ after relabeling the variables. The two possibilities for $H_{(R/J)_P}$ are either $(1,1,1)$ or $(1,2,0)$. In the first case, we have $q_i=0$ for all $i =1 \ldots 4$. It is also easy to show that if three of the $g_i's \in J$, then the fourth one is also in $J$. Hence, $(x,y,z,w)^3+(g_1,g_2,g_3,g_4) \subset J$. When ht$(I_5(N')) \geq 3$ mod $(x,y,z,w)$, with $N'$ the matrix defined in lemma \ref{lemma25}, the ideal on the left is unmixed and equal to $J$ by lemma \ref{lemma25}. This puts us in case {\color{blue} II-}\ref{height}.  Since $n^5, (a,b,e,f,k,l)((al-bk)^3, (af-be)^3, (el-fk)^3) \subset I_5(N')$, then ht$( I_5(N'))   \geq 3$ whenever ht$\left(I_2\begin{pmatrix} a&b\\e&f\\ k&l \end{pmatrix} \right)=2$. Hence, we need to study the case when ht$\left(I_2\begin{pmatrix} a&b\\e&f\\ k&l \end{pmatrix}\right)=1$. By \cite [Lemma 4.1] {HMMS13} we have three cases:\\
  Case 1. The matrix has the form: $\begin{pmatrix} a&b\\e&0\\ k&0 \end{pmatrix}$. If ht$(x,y,z,w,a,b,n,e,k) \geq 8$ then by lemma \ref{lemma26}, $J = (x, y, z,w)^3+ (ax+by, ex+nz, kx+nw, ew-kz)$ which puts us in case  {\color{blue} II-}\ref{height}. If ht$(x,y,z,w,a,b,n,e,k) =7$, we get back to case  {\color{blue} II-}\ref{vars}.\\
  Case 2. The matrix has the form: $\begin{pmatrix} a&b\\e&f\\ 0&0 \end{pmatrix}$. Hence by lemma \ref{lemma27},  $J =(x,y,z)^3+(w, ax+by, ex+fy+nz)$ whenever ht$(x,y,z,w,a,b,n,e,f) \geq 7$, which puts us again in cases {\color{blue} II-}\ref{height} or  {\color{blue} II-}\ref{vars}.\\
  Case 3. The matrix has the form: $\begin{pmatrix} a&b\\e&0\\ 0&e \end{pmatrix}$.  Since there are only three variables in the matrix then we are in case {\color{blue} II-}\ref{vars}. \\
  When $H_{(R/J)_P}=(1,2,0)$, $J$ contains a quadrics $g_i$ with $q_i \neq 0$. Since $J$ contains three quadrics, then $(x,y,z,w)g_l\in J$ where $g_l$ is the $4^{th}$ quadric. If ht$(I_5(N')) \geq 3$,  then $J=(x,y,z,w)^3 +(g_i, g_j, g_k)+ (x,y,z,w)g_l$ for $i \neq j \neq k \neq l$, by lemma \ref{lemma28} which puts us in case {\color{blue} II-}\ref{height}.
Hence, we study the case when  ht$\left(I_2\begin{pmatrix} a&b\\e&f\\ k&l \end{pmatrix}\right)=1$. As above, this puts us in {\color{blue} II-}\ref{height} or {\color{blue} II-}\ref{vars}.\\

  \underline{Case VI. $J:P = (x,y,z,w)^2+(ax+gy, ex+gz, ey-az, kx+ly+mz+nw) $}\\
  
  This case is similar to Case III.  Since we have  $(x, y, z,w)^3+ (x,y,z,w)(ax+gy, ex+gz, ey-az, kx+ly+mz+nw)\subset J$ and of multiplicity $6$, then there exist at least three additional quadrics. We may assume these quadrics to be three of the following $g_1= ax+gy +q_1, g_2=ex+gz+q_2, g_3= ey-az+q_3$ and $g_4=kx+ly+mz+nw +q_4$ with $q_1 \in (z,w)^2, q_2 \in (y,w)^2$, $q_3 \in(x,w)^2$ and $q_4=0$ after relabeling the variables. 
  When $H_{(R/J)_P}=(1,1,1)$, then $q_1=q_2=q_3=0$. Hence, by lemma \ref{lemma29}, $J =(x, y, z,w)^3+(ax+gy, ex+gz, ey-az, kx+ly+mz+nw)$ if ht$(x,y,z,w,a,g,e,k,l,m,n) \geq 8$, which puts us in case {\color{blue} II-}\ref{height}. If  ht$(x,y,z,w,a,g,e,k,l,m,n) \leq 7$, then we are back to case {\color{blue} II-}\ref{vars}.\\
  When $H_{(R/J)_P}=(1,2,0)$, there exists an $i$ such that $q_i \neq  0$. In that case, by lemma \ref{lemma30}, $J =(x,y,z,w)^3+(g_i, g_j, kx+ly+mz+nw)+(x,y,z,w)g_l$ for $i, j, l=1,\ldots 3$, which are cases {\color{blue} II-}\ref{height} or {\color{blue} II-}\ref{vars}.
  \end{proof}

\section {The projective dimension of $5$ quadric almost complete intersections}

In this section, we let  $I$ be an almost complete intersection generated by $5$ quadrics.  In \cite[Question 6.2]{HMMS131} and \cite[Question 10.2]{HMMS131}, the authors asked the following question:

\begin{question} Let $S$ be a polynomial ring and let $I$
be an ideal of $S$ generated by $n$ quadrics, having $ht( I) = h$. Is it true that
$pd(S/I) \leq h(n - h + 1)$?
\end{question}
We give an affirmative answer to their question for almost complete intersections generated by $5$ quadrics with small multiplicities. We show that, when $I$ has multiplicity $\leq 3$, then $\mbox{pd}(S/I) \leq 8$. When the multiplicity is less or equal than $2$, we make use of the inequality between the multiplicity and the Cohen-Macaulay defect, found in \cite[Theorem 2.5]{HT17}. If the multiplicity is three, then $I^{un}$ is one of the following types: $\langle 3;1 \rangle, \langle 1;3 \rangle, \langle 1,2;1,1 \rangle, \langle 1,1;1,2 \rangle$ and $ \langle 1,1,1;1,1,1\rangle$. We prove in every case, that pd$(S/I) \leq 8$.

\begin{notation}In this section, we use the following notation:
\begin{itemize}
\item $S$ is a polynomial ring over an algebraically closed field $k$,
\item $I = (q_1, q_2, q_3, q_4, q_5)$ where $q_i$ are polynomials of degree $2$ and height $(I)=4$
\item $L = (q_1, q_2, q_3, q_4):I=(q_1, q_2, q_3, q_4):I^{un}$
\end{itemize}
\end{notation}

\begin{corollary} Let $I$ be an almost complete intersection generated by $5$ quadrics with $e(S/I) \leq 2$, then pd$(S/I) \leq 6$. 
\end{corollary}

\begin{proof} If $e(S/I) \leq 2$, then by \cite[Theorem 2.5]{HT17}, $e(S/I) \geq dim (S/I) - depth (S/I)$. Hence $2 \geq e(S/I) \geq \mbox{pd}(S/I) -4$, which implies the assertion.
\end{proof}

Before we move to multiplicity three, we need the following lemmas.
\begin{lemma} \label{linearform} If $I^{un}$ or $L = (q_1, q_2, q_3, q_4): I$ contains two linear forms then pd$(S/I) \leq 6$
\end{lemma}

\begin{proof} The proof goes along the same line as the proof of Lemma 3.4 in \cite{HMMS13}. We suppose $x$ and $y$ are two linear forms contained in $I^{un}$. Since  $I \subset (I, x,y) \subset I^{un}$, then ht$(I, x) = 4$. 
Hence, after relabeling the quadrics,  we may assume $x, y, q_1, q_2$ form a regular sequence. Let $L' = (x, y, q_1, q_2) : I = (x, y, q_1, q_2) : I^{un}$.
Since $L'$ contains a complete intersection of two linear forms and two quadrics then $e(S/L') \leq 4$. If $e(S/L')=4$ then $L '= (x,y,q_1,q_2)$ and $L'$ is CM, hence pd$(S/I) \leq 5$ by Lemma \ref{pdim-linkage}. If $e(S/L')=1$ then $L$ is CM by theorem \ref{SN} and pd$(S/I) \leq 5$ again by Lemma \ref{pdim-linkage}.
If $e(S/L')=2$, then write $L'=(x,y)+ L''$, where $x$ and $y$ are regular on $L''$, $L''$ unmixed of multiplicity $e(S/L'')=2$ and height ht$(L'')=2$.  By \cite[Proposition 11]{E07}, pd$(S/L'') \leq 3$ which yields to pd$(S/L') \leq 5$ and hence pd$(S/I) \leq 6$. Finally if $e(S/L')=3$ consider $J = (x, y, q_1, q_2): L'$, by Theorem \ref{PS},  $e(S/J)= e(S/(x,y,q_1,q_2))- e(S/L')=4-3=1$. So $J$ is CM by Theorem \ref{SN}. Hence $L'$ is CM and pd$(S/L')= 3$, which implies that pd$(S/I) \leq 4$. 
\end{proof}

\begin{lemma}\label{KJ} Suppose $I \subset K \cap J$ where $K = K'+(q)$, ht$(K)=$ ht$(J)=4$, ht$(K')=3$ and $q$ a quadric. Then there exists a quadric $q' \in K \cap J$ such that $K=K'+(q')$.
\end{lemma}
\begin{proof}  For each $i= 1 \cdots 5$, write $q_i= f_i+\alpha_i q$ where $f_i \in K'$ and $\alpha_i \in k$. If $\alpha_i=0$ for all $i$, one has $I \subset K'$ which implies ht$(K') \geq 4$ and gives a contradiction. Hence, one may assume $\alpha_i \neq 0$ for some $i$. Take $q'=q_i$.
\end{proof}

In the next theorems, we suppose $e(S/I)=3.$

\begin{theorem} If $I^{un}$ is of type $\langle 3;1 \rangle$, then pd$(S/I) \leq 8$. 
\end{theorem}

\begin{proof} If $I^{un}$ is of type $\langle 3;1 \rangle$, then it contains a linear form say $x$ by Proposition \ref{mult-height}. We write $I^{un}=(x)+I'$, where $I'$ is unmixed of type $\langle 3;1 \rangle$ and height three. By \cite[Theorem 2.10]{HMMS13}, $I'$ contains another linear form. Hence pd$(S/I) \leq 5$ by Lemma \ref{linearform}.
\end{proof}

\begin{theorem}\label{primary} If $I^{un}$ is of type $\langle 1;3 \rangle$, then pd$(S/I) \leq 8$.   \end{theorem}

\begin{proof} If $I^{un}$ is of type $\langle 1;3 \rangle$, then $I^{un}$ is one of proposition \ref{mult3}. When $I^{un}$ contains two linear forms then 
by lemma \ref{linearform}, pd$(S/I) \leq 6$.  This cover cases I-$(i)$ to $(iv)$. In cases I-$(v)$ to $(xi)$, I-$(xiii)$ and II-$(xii)$, all quadrics in $I^{un}$ are expressible in terms of at most $8$ variables. Hence by lemma \ref{extend},
 pd$(S/I) \leq 8$. When the quadrics of $I^{un}$ generate an ideal of height at most three, we get a contradiction since ht$(I)=4$. In the cases II-$(i)$ to $(x)$, we note by Lemmas \ref{lemma5}, \ref{lemma10}, \ref{lemma11}, \ref{lemma16}, \ref{lemma17}, \ref{lemma19} and \ref{lemma20}, that any ideal $L$ directly linked to $I^{un}$ satisfies pd$(S/L) \leq 6$ . Also by Lemmas \ref{lemma18} and \ref{lemma21}, we get pd(Corker$(\partial_5^*)) \leq 6$ where $\partial_i$ is the $i^{th}$ differential map in the resolution of $I^{un}$. Hence pd$(S/I) \leq 7$ by Lemma \ref{pdim-linkage} and Proposition \ref{coker}.
\end{proof}

\begin{theorem}If $I^{un}$ is of type $\langle 1,2;1,1 \rangle$, then pd$(S/I) \leq 6$.
\end{theorem}

\begin{proof} If $I^{un}= (x,y,z,w) \cap (u,v,s,q)$ where $q\in S_2$ and the rest in $S_1$. We may assume that $q \in (x,y,z,w)$ by Lemma \ref{KJ}, say $q=ax+by+cz+dw$. If ht $(x,y,z,w,u,v,s) \leq 5$, then $I^{un}$  contains two linear forms and pd$(S/I) \leq 5$ by Lemma \ref{linearform}. If ht $(x,y,z,w,u,v,s)=6$ then $I^{un}$ contains a linear form say $x$. We write $I^{un}=(x)+(y,z,w)\cap (v,s,q)$, with $q=by+cz+dw$. This composition is the sum of a linear form and a height three unmixed ideal of type $\langle 1,2;1,1 \rangle$. By the proof of \cite[Lemma 7.2]{HMMS13}, either $(y,z,w)\cap (v,s,q)$ contains a linear form or pd(Corker$(\partial_4^*)) \leq 5$ where $\partial_i$ is the $i^{th}$ differential map in the resolution of $(y,z,w)\cap (v,s,q)$. In the first case we get pd$(S/I) \leq 5$ by Lemma \ref{linearform}, and in the second case pd(Corker$(\partial_5^*)) \leq 6$ where $\partial_i$ is the $i^{th}$ differential map in the resolution of $I^{un}$. Hence pd$(S/I) \leq 6$ by Proposition \ref{coker} . We may assume that ht$(x,y,z,w,u,v,s)=7$. We have $I^{un}= (q, xu, xv, xs, yu, yv, ys, zu, zv, zs, wu, wv, ws)$, and we let $L= (xu, yv, sz,q): I^{un}$. By Lemma \ref{lemma31} , pd$(S/L) \leq 5$, and hence  pd$(S/I) \leq 6$ by Lemma \ref{pdim-linkage}.

\end{proof}
\begin{theorem}If $I^{un}$ is of type $\langle 1,1;1,2 \rangle$, then pd$(S/I) \leq 8$.
\end{theorem}

\begin{proof} Suppose that $I^{un}$ is $(u,v,s,t) \cap L_2$ where $L_2$ is as Proposition \ref{multiplicity2}. \\
 $\bullet$ \underline {Case 1}. If $L_2 =  (x,y,z,w^2)$, then $I^{un}$ can be expressed by at most $8$ variables,  and pd$(S/I) \leq 8$ by Lemma \ref{extend}. \\
$\bullet$ \underline {Case 2}.  If $L_2= (x,y)+(z,w)^2+ (az+bw )$ with ht$(x,y,z,w,a,b)=6$. When ht$(u,v,s,t,x,y)=4$, then $I^{un}$ contains two linear forms and pd$(S/I) \leq 6$ by Lemma \ref{linearform}.  If ht$(u,v,s,t,x,y)=5$, then we can write $I^{un}=(x)+ (v,s,t)\cap ((y)+(z,w)^2+(az+bw))$ with ht$(v,s,t,y)=4$. The case of $(v,s,t)\cap ((y)+(z,w)^2+(az+bw))$ was studied in the proof of in \cite [Lemma 7.3]{HMMS13}.  We get either the quadrics of $(v,s,t)\cap ((y)+(z,w)^2+(az+bw))$ generate a height two ideal, or they are expressed with at most $6$ variables. Hence, the quadrics of $I^{un}$ generate a height three ideal, or they are expressed with at most $7$ variables. So pd$(S/I) \leq 7$. We then assume that ht$(u,v,s,t,x,y)=6$. We may assume that the degree of $a$ and $b$ is one and $az+bw \in (u,v,s,t) \cap L_2$, or else the quadrics of $I^{un}$ will be expressed with at most $8$ variables. Hence, if ht$(u,v,s,t,z,w)= 6$ then $(u,v,s,t) \cap L_2= (az+bw)+(u,v,s,t)(x,y,z^2, zw, w^2)$, and all quadrics of $I^{un}$ generate an ideal of height at most three.  This contradicts our assumption about $I$.  If ht$( u,v,s,t,z,w) \leq 5$, then we may assume that $z=u$. Hence $au+bw \in (u,v,s,t) \cap L_2 $. We must have $w \in (v,s,t)$ or $b \in (u,v,s,t)$. Hence ht$(x,y,z,w,u,v,s,t,a,b) \leq 8$, and all quadrics in $I^{un}$ can be expressed with at most $8$ variables. So pd$(S/I) \leq 8$.\\
 $\bullet$ \underline {Case 3}. If  $L_2= (x)+ (y,z,w)^2+ (by+cz+dw, ey+fz+gw)$  with {\small ht$\left(x,y,z,w,I_2\begin{pmatrix}b&c&d\\ e&f&g \end{pmatrix}\right) \geq 6$.} 
\normalsize We may assume that the degree of $b,c,d,e,f, g$  is $1$, and  $by+cz+dw, ey+fz+gw \in (u,v,s,t) \cap L_2 $, or else the quadrics of $I^{un}$ will be expressed with at most $8$ variables.  If ht$(x,y,z,w,u,v,s,t)=8$, then $(u,v,s,t) \cap L_2= (by+cz+dw, ey+fz+gw)+(u,v,s,t)(x,y^2, yz,yw, z^2, zw, w^2)$, and the quadrics of $I^{un}$ generate an ideal of height three. If ht$(x,y,z,w,u,v,s,t) \leq 7$ and $x\in(u,v,s,t)$ (we take  $x=u$), then we may write $I^{un}=(x)+  (by+cz+dw, ey+fz+gw)+(v,s,t)\cap (y,z,w)^2 $. By the proof of \cite[Lemma 7.3]{HMMS13}, we get that all quadrics of $I^{un}$ are expressed with at most $7$ variables. We suppose next that ht$(x,y,z,w,u,v,s,t) \leq 7$ and $x\notin(u,v,s,t)$. Suppose first that ht$(x,y,z,w,u,v,s,t) =7$ and  $y=u$. We let $I= (f_1l_1, f_2l_2, f_3l_3,by+cz+dw, ey+fz+gw)$ with $e(S/I)=3$, where $f_i=\alpha_ix+\beta_iy$, $\alpha_i, \beta_i \in k$ and $l_i$ are linear forms. But $I+(x,y)=(x,y,cz+dw, fz+gw)$ and ht$(x,y,I_2\begin{pmatrix}b&c&d\\ e&f&g \end{pmatrix}) =4$. Hence ht$(S/(I+(x,y)))=4 $ and $e(S/(I+(x,y)))=4$. This is impossible, since $I \subset (I+(x,y))$ and $e(S/I)=3$.   If ht$(x,y,z,w,u,v,s,t) = 6$, then we take  $y=u$ and $z=v$. We have $by+cz+dw, ey+fz+gw \in (u,v,s,t)$.  Hence $d \in (s)$ and $g=0$. The case $d=g=0$ cannot happen due to the height restriction on the matrix of minors $I_2$. So, $I^{un}= (xy, xz, xs, xt, y^2, yz,yw, z^2, zw,  w^2t, w^2s,  by+cz+ws, ey+fz)$. Since ht$(x,y,z,w,s,t,I_2\begin{pmatrix}b&c&d\\ e&f&g \end{pmatrix}) \geq 7$ then, by Lemma \ref{lemma32}, pd$(S/L) =5$ where $L$ is directly linked to $I^{un}$. This implies  that pd$(S/I) \leq 6$ by Lemma \ref{pdim-linkage}.
  If ht$(x,y,z,w,u,v,s,t)=5$, then we may assume further that $w=s$. Hence  $I^{un}=(xt, xy, xz, xw, y^2, yz, yw, z^2, zw, w^2, by+cz+dw, ey+fz+gw)$. This case can be treated in a similar manner as whenever ht$(x,y,z,w,u,v,s,t)=7$. Finally, if ht$(x,y,z,w,u,v,s,t) = 4$, then $x \in(u,v,s,t)$, and this case is already treated above.\\
$\bullet $ \underline {Case 4}.  Assume that $L_2= (x, y, z,w)^2+ (ax+by+cz+dw, ex+fy+gz+hw, kx+ly+mz+nw)$  with ht$(x,y,z,w,I_3(M))\geq 6$, where {\small $M=\begin{pmatrix} a&b&c&d\\ e&f&g&h\\k&l&m&n \end{pmatrix}$}. \normalsize We use a similar argument as the previous cases. If ht$(x,y,z,w,u,v,s,t)=8$, then the quadrics of  $I^{un}$ generate a height at most three.  If ht$(x,y,z,w,u,v,s,t)=6$ or $7$,  then both cases will be treated the same way as whenever ht$(x,y,z,w,u,v,s,t) = 7$ in Case 3.  If ht$(x,y,z,w,u,v,s,t)=5$, then we may assume that $x=u$, $y=v$  and $z=s$. Since at least two of $ax+by+cz+dw, ex+fy+gz+hw, kx+ly+mz+nw \in (x,y,z,t)$, then we may assume that $d=t$ and $h=0$ after a linear change of variables. We get either $(u,v,s,t) \cap L_2=(ax+by+cz+tw, ex+fy+gz, kx+ly+mz)+(x,y,z,t)\cap (x,y,z,w)^2$, or $ (ax+by+cz+tw, ex+fy+gz)+(x,y,z,t)\cap[(x,y,z,w)^2+(kx+ly+mz+nw)]$.  In the first case, since ht$(x,y,z,w, I_3(M))=6$, we get ht$\begin{pmatrix} e&f&g\\ k&l&m \end{pmatrix} \geq 2$ mod $(x,y,z,w)$. By Lemma \ref{lemma33} we obtain pd($S/L)=5$ for an ideal $L$ directly linked to $I^{un}$. Hence, pd$(S/I) \leq 6$ by Lemma \ref{pdim-linkage}. In the second case, we suppose $n \notin t$ or else we are back to the first case. Since ht$(I_3(m)) \geq 2$ mod $(x,y,z,w)$ then, by Lemma \ref{lemma34}, pd$(S/L)=5$ for an ideal $L$ directly linked to $I^{un}$. This implies that pd$(S/I) \leq 6$ by Lemma \ref{pdim-linkage}. If ht$(x,y,z,w,u,v,s,t)=4$, then $x=u$, $y=v$, $z=s$ and $w=t$. $I^{un}= (ax+by+cz+tw, ex+fy+gz, kx+ly+mz)+(x,y,z,w)^2$ which is $\mathfrak p$-primary. Hence by Proposition \ref{primary}, pd$(S/I) \leq 7$. \\
  $\bullet $ \underline {Case 5}. If $L_2= (x,y,z,w)^2+(ax+by, ex+fy+nz, kx+ly+nw, (af-be)w-(al-kb)z)$ with ht$(a,b,n)=3$,  then in most cases the quadrics of $I^{un}$ generate an ideal of height at most three, or we can use the same method as whenever ht$(x,y,z,w,u,v,s,t) = 7$ of Case 3 and show a contradiction. The only remaining cases to study is whenever $(u,v,s,t)\cap L_2=(ax+by, kx+ly+nw)+(x,y,z,t)\cap ((x,y,z,w)^2, ex+fy+nz)$ or $(ax+by, kx+ly+mz+nw)+(x,y,s,w)\cap ((x,y,z,w)^2, ex+fy+nz)$. Both cases will be treated the same way, so we will prove only one of them. Suppose   $I^{un}=(ax+by, kx+ly+nw)+(x,y,z,t)\cap ((x,y,z,w)^2, ex+fy+nz)=(x,y,z)^2 +(wx,wy,wz, w^2t,ax+by, kx+ly+nw, ex+fy+nz)$, then by Lemma \ref{lemma35}, pd$(S/L)=5$ for an ideal $L$ directly linked to $I^{un}$. Hence pd$(S/I)\leq 6$.\\
 $\bullet $ \underline {Case 6}. Finally, let $L_2 = (x,y,z,w)^2+(ax+gy, ex+gz, ey-az, kx+ly+mz+nw)$. We may suppose that   $kx+ly+mz+nw \in (u,v,s,t)$ or else the quadrics of $I^{un}$ will be expressed in terms of  at most $8$ variables. If ht$(x,y,z,w,u,v,s,t)=8$, then the quadrics of $I^{un}$ can be expressed with at most $8$ variables.  If ht$(x,y,z,w,u,v,s,t)=7$, then we may assume that $x=u$. This case will be solved the same way as whenever ht$(x,y,z,w,u,v,s,t) = 7$ of Case 3.   If ht$(x,y,z,w,u,v,s,t)=6$, then we may assume that $x=u$ and $y=v$. In this case, the quadrics of $I^{un}$ will generate an ideal of at most three which is a contradiction.  If ht$(x,y,z,w,u,v,s,t)=5$, then we may assume that $x=u$, $y=v$, and $z=s$ or $w=t$. The only cases to study are $I^{un}=(ax+gy, ex+gz, ey-az, kx+ly+mz+nw)+ (x,y,z,t)\cap (x,y,z,w)^2$ or $I^{un}=(ax+gy, ex+gz, ey-az, kx+ly+mz+nw)+ (x,y,s,w)\cap(x,y,z,w)^2$. Both cases will be treated the same way, so we prove the first one. In the first case, $I^{un}=(ax+gy, ex+gz, ey-az, kx+ly+mz+nw, x^2, xy, xz, xw, y^2, yz, yw, z^2, zw, w^2t )$. By Lemma  \ref{lemma36}, pd$(S/L) \leq 5$  for an ideal $L$ linked to $I^{un}$. Hence pd$(S/I) \leq 6$. \end{proof}

\begin{theorem}If $I^{un}$ is of type $\langle 1,1,1;1,1,1 \rangle$, then pd$(S/I) \leq 6$.
\end{theorem}

\begin{proof} Write $I^{un}= \cap_{i=1}^3L_i$ where $L_i=(x_i, y_i, z_i, w_i)$. If ht$(L_i+ L_j)=8$ for any $i \neq j$, then all quadrics of $J$ will be expressed with at most $8$ variables which are $x_i, y_i, z_i, w_i, x_j, y_j, z_j, w_j$. Otherwise ht$(L_i+ L_j)\leq 7$ for any $i \neq j$. In that case, either all generators are expressed with at most $8$ variables or $I^{un}$ contains a linear form. In the latter case, we write $I^{un}= (x)+ I'$ where $I'=\cap_{i=1}^3L_i'$ and $L_i'=( y_i, z_i, w_i)$. Hence by the proof of \cite[Lemma 7.4]{HMMS13}, either $I'$ contains a linear form or the quadrics of $I'$ are expressed with at most $6$ variables or all generators are expressed with at most $6$ variables. So either $I^{un}$ contains two linear forms or the quadrics of $I^{un}$ are expressed with at most $7$ variables or the generators of $I^{un}$ are expressed with at most $8$ variables. Hence, pd $(S/I)\leq 8$ by Lemma \ref{linearform} or Corollary \ref{extend}.
\end{proof}

\section* {Appendix: Resolution of Primary Ideals}
We follow the same techniques as  of \cite[Appendix A]{HMMS13}. We resolve the unmixed primary ideals generically and check the exactness of the resolution by using the Buchsbaum-Eisenbud exactness criteria. For that, if $\mathbb{F}$ is the resolution of $S/I$ and $\partial_i$ denotes the $i^{th}$ differential map, then it suffices to check that ht$(I_{r_j}(\partial_j)) \geq j$ for all $j$, where $r_j = \sum_{i=j}^p(-1)^{p-i}$ rank$(F_i)$ and $I_{r}(\partial_j)$ is the $r \times r$ minors of the matrix associated to the $j^{th}$ differential. Further if ht$(I_{r_j}(\partial_j)) \geq j+1$ for $j >$ ht$(I)$ then the ideal is unmixed by \cite [Proposition 2.4] {HMMS16}. In the case when ht$(I)=4$, it suffices to show
$$\mbox{ht}(I_{r_j}(\partial_j) )\geq \left\{\begin{array}{ccc} j& \mbox{for}& j = 1, 2,3 ,4 \\ j+1& \mbox{for}& j  \geq 5 \end{array} \right.$$
We give a complete proof for lemma \ref{lemma1}, the rest is done in a similar manner. The computations were done by using the computer algebra Macaulay 2 \cite{GS}. 

\begin{customlemma}{A.1}\label {lemma1}If $$J = (x, y, z, w)^2+ (ax+by+cz+dw, ex+fy+gz+hw, kx+ly+mz+nw)$$ where ht$\left(x, y, z, w,I_3\begin{pmatrix} a&b&c&d\\
e&f&g&h\\k&l&m&n \end{pmatrix} \right)=6$  then $J$ is $(x,y,z,w)$-primary, pd$(S/J) =5$ and $e(S/J)= 2$. 
\end{customlemma}
 \begin{proof} First, we let $N= (x,y,z,w)^2$ with a minimal free resolution obtained by using the Eagon- Northcott complex $$0 \to S^4  \stackrel{D_4}{\longrightarrow}S^{15} \stackrel{D_3}{\longrightarrow} S^{20} \stackrel{D_2}{\longrightarrow}  S^{10} \stackrel{D_1}{\longrightarrow}S \to S/N \to 0 $$ 
 
 We then consider the complex $$0 \to R^3 \stackrel{\partial_5}{\longrightarrow}R^{16}\stackrel{\partial_4}{\longrightarrow}R^{33} \stackrel{\partial_3}{\longrightarrow}R^{32} \stackrel{\partial_2}{\longrightarrow}R^{13} \stackrel{\partial_1}{\longrightarrow}R$$ with 
 $\partial_1 = (ax+by+cz+dw, ex+fy+gz+hw, kx+ly+mz+nw, D_1)$
 $$\partial_2=\tiny{\left( \begin{array}{ccccccccccccccccc} 
     &-a&0&0& 0 & -e &0 & 0 & 0&-k&0 &0&0 \\
      &-b&-a&0&0&-f&-e& 0&0& -l& -k &0 & 0 \\
   & 0&-b&0 & 0&0 &-f&0&0&0&-l&0 &0\\
     & -c&0&-a&0&-g&0&-e&0&-m&0&-k& 0 \\
 D_2   &0&-c&-b&0&0&-g&-f& 0&0&-m&-l&0\\
     &0&0&-c&0&0&0&-g&0&0 &0 & -m &0\\
    &-d&0&0&-a&-h& 0 & 0&-e&-n&0&0  &-k\\
    & 0&-d&0 &-b& 0 &-h& 0 & -f& 0&-n& 0 & -l\\
     &0&0&-d &-c& 0&0 & -h& -g &0 & 0 & -n& -m \\
 &0&0&0&-d& 0&  0&  0&  -h& 0 & 0&  0&  -n\\
  0 \ldots 0 &x&y&z&w&0 &0 &0 &0 & 0 &0 &0 & 0 \\
      0 \ldots  0&0&0&0& 0 & x & y & z&w&0& 0 &0 & 0\\
      0 \ldots 0&0&0&0&0 &0 &0 & 0 & 0 & x & y & z&w  \end{array} \right)}$$
$$\partial_3=\Tiny{\left( \begin{array}{cccccccccccccccccccccc}
&a& 0&0&0&0&0&e&0& 0&  0 & 0 & 0 & k & 0 & 0 & 0 & 0 & 0\\
   &b&0&0&0&0&0&f&0 &0&0 &0 & 0 &l&   0&  0&  0&  0 & 0  \\
     &0&a&0&0&0&0&0&e& 0&  0 & 0 & 0&  0&  k&  0&  0&  0 & 0\\
&0&b&a&0&0&0&0&f &e&  0&0&  0&  0 & l&  k & 0&  0 & 0  \\
  &  c&b&0&0&0&0&g&f & 0&  0 & 0 & 0 & m & l&  0&  0&  0 & 0\\
  &0&0&b&0&0&0&0  &0& f& 0 & 0&  0 & 0 & 0 & l & 0 & 0 & 0\\
    &0&c&0&0&0&0&0&g & 0&  0 & 0 & 0 & 0 & m & 0 & 0&  0&  0  \\
  &0&0&c&0&0&0&0&0 & g & 0 & 0&  0  &0 & 0 & m & 0 & 0 & 0\\
   &0&0&0&a&0&0&0&0 &  0 & e & 0 & 0 & 0 & 0 & 0 & k&  0 & 0\\
   &0&0&0&b&a&0&0&0& 0 & f&  e & 0 &0 & 0&  0 & l & k & 0\\
   &d&0&0&b&0&0&h&0 &0 & f & 0 & 0 & n & 0 & 0 & l & 0 & 0\\
    & 0&0&0&0&b&0&0&0 &0 & 0 & f & 0 & 0  &0 & 0&  0&  l & 0\\
D_3     & 0&0&0&c&0&a&0&0&0&  g & 0 & e & 0 & 0&  0 & m & 0&  k\\
      &0&d&0&c&0&0&0&h&0&  g & 0&  0&  0&  n & 0 & m & 0 & 0\\
      &0&0&0&0&c&b&0&0&0&0 &g & f& 0& 0&  0&  0 & m & l \\
        &0&0&d&0&c&0&0&0& h & 0 & g&  0 & 0&  0 & n  &0 & m & 0  \\
  &  0&0&0&0&0&c&0&0 & 0 & 0 & 0  &g&  0 & 0&  0&  0&  0 & m  \\
  &0&0&0&d&0&0&0&0 &0  &h & 0  &0&  0&  0 & 0  &n & 0  &0  \\
   &0&0&0&0&d&0&0&0& 0 & 0 & h & 0 & 0 & 0 & 0 & 0 & n & 0 \\
      &0&0&0&0&0&d&0&0&0 & 0 & 0 & h & 0 & 0 & 0 & 0 & 0 & n \\
 0 \dots 0  &-y&-z &0 &-w& 0&0&0&0& 0 & 0  &0 & 0&  0&  0 & 0  &0 & 0 & 0  \\
 \vdots   & x&0&-z&0&-w&0&0&0 & 0 & 0 & 0 & 0 & 0 & 0&  0&  0&  0&  0 \\
   &0&x&y&0&0&-w&0&0&0&  0 & 0 & 0 & 0 & 0 & 0 & 0&  0 & 0 \\
     &0&0&0&x&y&z&0&0& 0 & 0 & 0&  0 & 0 & 0&  0 & 0 & 0 & 0 \\
  &0&0&0&0&0&0&-y&-z & 0 & -w &0 & 0  &0 & 0&  0&  0&  0  &0  \\
    &0&0&0&0&0&0&x& 0& -z& 0  &-w &0 & 0 & 0&  0 & 0 & 0&  0  \\
  & 0&0&0&0&0&0&0&x& y& 0 & 0 & -w &0 & 0&  0  &0  &0 & 0  \\
    &0&0&0&0&0&0&0&0  &0 & x& y & z&  0 & 0&  0  &0 & 0  &0  \\
     &0&0&0&0&0&0&0&0& 0&  0 & 0&  0 & -y& -z &0&  -w& 0  &0  \\
     & 0&0&0&0&0&0&0&0 &0&  0&  0&  0 & x & 0 & -z& 0 & -w& 0 \\
  & 0&0&0&0&0&0&0&0& 0& 0 & 0 & 0 & 0 & x & y & 0 & 0 & -w\\
0 \dots 0    &0&0&0&0&0&0&0&0&  0&  0&  0 & 0  &0  &0&  0  &x&  y&  z\\
 \end{array} \right)}$$
 
  $$\partial_4=\Tiny{\left( \begin{array}{ccccccccccccccccc} 
     &  -a& 0  &0 & 0 & -e &0 & 0 & 0&  -k &0&  0 & 0  \\
     & -b& 0 & 0 & 0 & -f &0 & 0  &0&  -l &0 & 0  &0  \\
     &  -c &0  &0  &0 & -g& 0 & 0&  0 & -m& 0  &0&  0  \\
     &  0  &-a &0&  0 & 0  &-e &0& 0  &0 & -k &0&  0  \\
     &  0 & -b& 0 & 0 & 0&  -f &0  &0  &0&  -l &0&  0  \\
     &  0 & 0 & -a &0  &0&  0&  -e &0&  0 & 0  &-k& 0  \\
     &  0  &c & 0&  -a& 0&  g&  0  &-e &0  &m&  0  &-k \\
     & 0 &-c &-b &0&  0 & -g &-f& 0&  0  &-m& -l& 0  \\
     & -d &-c &0  &0 & -h& -g &0  &0 & -n &-m &0 & 0  \\
     & 0 & 0  &0&  -b& 0  &0  &0&  -f &0&  0&  0&  -l \\
  D_4   &  0  &0 & -c &0 & 0&  0&  -g& 0&  0 & 0 & -m& 0  \\
      &  0&  0  &0 & -c& 0&  0  &0  &-g& 0&  0&  0 & -m \\
     &  0  &-d &0 & 0 & 0 & -h &0  &0&  0 & -n &0 & 0  \\
     & 0 & 0 & -d& 0&  0  &0  &-h &0 & 0&  0&  -n &0  \\
      &  0 & 0  &0&  -d& 0 & 0&  0 & -h &0&  0  &0  &-n \\
     0 \ldots 0 &  z  &w & 0  &0  &0&  0&  0&  0 & 0  &0&  0&  0  \\
     \vdots  &  -y& 0  &w  &0 & 0&  0&  0 & 0&  0&  0  &0  &0  \\
      &  x&  0 & 0&  w & 0 & 0  &0  &0  &0  &0&  0 & 0  \\
     &  0  &-y &-z& 0 & 0 & 0 & 0  &0&  0&  0&  0&  0  \\
     &  0 & x  &0&  -z& 0&  0  &0 & 0 & 0&  0 &0  &0  \\
     &  0 & 0 & x  &y&  0 & 0 & 0 & 0&  0&  0  &0&  0  \\
      & 0&  0&  0&  0&  z&  w&  0 & 0 & 0  &0 & 0 & 0  \\
    & 0 & 0 & 0&  0 & -y &0 & w&  0  &0 & 0&  0&  0  \\
    & 0  &0&  0&  0&  x  &0 & 0 & w  &0  &0 & 0&  0  \\
    &  0&  0&  0  &0 & 0  &-y &-z &0  &0  &0 & 0 & 0  \\
   & 0  &0&  0  &0  &0 & x&  0 & -z &0&  0 & 0  &0  \\
    & 0  &0 & 0 & 0 & 0&  0 & x  &y&  0&  0 & 0  &0  \\
    &  0  &0  &0 & 0 & 0  &0  &0&  0&  z&  w & 0 & 0  \\
     & 0  &0 & 0 & 0&  0&  0 & 0 & 0 & -y& 0 & w & 0 \\
     &  0 & 0 & 0 & 0&  0&  0&  0&  0 & x & 0 & 0 & w  \\
      &  0  &0 & 0&  0&  0&  0  &0 & 0 & 0 & -y &-z &0  \\
    &  0 & 0 & 0 & 0 & 0 & 0 & 0 & 0 & 0 & x & 0 & -z \\
   0 \ldots 0  &  0& 0 & 0 & 0 & 0&  0&  0 & 0 & 0 & 0&  x&  y  
 \end{array} \right)}$$

and  $\partial_5=\tiny{\left( \begin{array}{ccccccccccccc}  
   a & e & k \\
    b & f & l \\
     c & g & m  \\
     d & h & n \\
      -w& 0 & 0 \\
       z & 0 & 0 \\
       -y &0 & 0 \\
       x & 0 & 0 \\
      0 & -w &0  \\
      0 & z & 0 \\
      0 & -y &0 \\
       0 & x & 0 \\
       0 & 0 & -w \\
      0 & 0 & z \\
      0 & 0 & -y \\
      0 & 0 & x  
  \end{array} \right)}$
  
  We note that $x^2 \in I_1(\partial_1)$ and $x^{12}, y^{12} \in  I_{12}(\partial_2)$. Further we note that  $x^{20}, y^{20}$ and $z^{20}\in  I_{20}(\partial_3)$ and $x^{13}, y^{13}, z^{13}$ and $w^{13} \in I_{13}(\partial_4)$. Finally $x^3, y^3, z^3, w^3$ and $I_3 \begin{pmatrix} a&b&c&d\\
e&f&g&h\\k&l&m&n \end{pmatrix} \in I_3(\partial_5)$. \\
Since ht$\left( I_3 \begin{pmatrix} a&b&c&d\\
e&f&g&h\\k&l&m&n \end{pmatrix} \right) \geq 2 $ mod $(x,y,z,w)$ then the complex is exact and resolves $J$. So $J$ is unmixed with pd$(S/J)=5$. Further $\lambda(S_{\mathfrak p} / J_{\mathfrak p})=2$ and  $\sqrt J = (x,y,z,w)$. Hence, $J$ is $(x,y,z,w)$-primary with multiplicity $e(S/J)=2$.

 \end{proof}
 
 \begin{customlemma}{A.2} \label{lemma2} If $$J = (x,y,z,w)^2+(ax+by, ex+fy+nz, kx+ly+nw, (al-bk)z-(af-be)w )$$where ht$\left(x,y,z,w,n,I_2\begin{pmatrix} a&b\\ e&f \\ k&l\end{pmatrix}\right)=7$, then $J$ is $(x,y,z,w)$-primary with pd$(S/J)=6 $ and $e(S/J)=2$. 
 \end{customlemma}
 
 \begin{customlemma}{A.3} \label{lemma3} If $$J = (x,y,z,w)^2+(ax+by, ex+bz, ey-az, kx+ly+mz+nw)$$ with ht$(x,y,z,w,a,b,e)=8$, then $J$ is $(x,y,z,w)$- primary, pd$(S/J)=6 $ and $e(S/J)=2$. \end{customlemma}
 
 \begin{customlemma}{A.4} \label{lemma4} If $$J =(y,z,w)^2+(x, by+cz+dw)$$ with ht$(x,y,z,w,b,c,d)\geq 6$, then $J$ is $(x,y,z,w)$-primary, pd$(S/J)=5$ and $e(S/J)=3$.
  \end{customlemma}
  
   \begin{customlemma}{A.5} \label{lemma5} If $$J =(x,y,z,w)^2+(ax+ by+cz+dw, ex+fy+gz+hw)$$ with ht$\left(I_2 \begin{pmatrix}a&b&c&d\\ e&f&g&h \end{pmatrix}\right) \geq 2$ mod $(x,y,z,w)$, then $J$ is $(x,y,z,w)$-primary, pd$(S/J)=5$ and $e(S/J)=3$.\\
   
       Furthermore\\
   $ L= (x^2, y^2, z^2, w^2):J=(x^2, y^2, z^2, w^2, xyzyw, x(yz(bg-cf)- yw(bh-df)+zw(ch-dg), y(xz(ag-ce)-xw(ah-de)+zw(ch-dg),z(xy(af-be)-xw(ah-de)+yw(bh-fd)), w(xy(af-be)-xz(ag-ce)+yz(bg-cf))$ and pd$(S/L)=6$.
  \end{customlemma}
  
     \begin{customlemma}{A.6} \label{lemma6} If $$J=(x,y,z,w)^2+(ax+by, ex+bz, ey-az)$$ with ht$(x,y,z,w,a,b,e)=7$, then $J$ is $(x,y,z,w)$-primary, pd$(S/J)=6$ and $e(S/J)=3$.
  \end{customlemma}
  
    \begin{customlemma}{A.7} \label{lemma7} If $$J=(x,y,z^2, wz, w^3, w^2+az)$$ then $J$ is $(x,y,z,w)$-primary with pd$(S/J)=4$ and e$(S/J)=3$.
    \end{customlemma}
    
    \begin{customlemma}{A.8} \label{lemma8} If   $$J=(x)+(y,z)^2+( wy, wz, w^2, by+cz+dw)$$ with ht$(x,y,z,w,b,c,d) \geq 6$  then $J$ is $(x,y,z,w)$-primary with pd$(S/J)=5$ and e$(S/J)=3$.
    \end{customlemma}
     
    \begin{customlemma}{A.9} \label{lemma9}   If $$J=(x)+(y,z)^2+( wy, wz, w^2+by+cz, ey+fz)$$ with ht$(x,y,z,w,e,f)=6$  then $J$ is $(x,y,z,w)$-primary with  pd$(S/J)=5$ and e$(S/J)=3$.
         \end{customlemma}
         
    \begin{customlemma}{A.10} \label{lemma10}  If   $$J=(x,y,z)^2+(wx,wy,wz,w^2+ ax+by+cz, ex+fy+gz, kx+ly+mz)$$ with ht$\left(I_2 \begin{pmatrix}e&f&g\\ k&l&m\\ \end{pmatrix}\right)\geq 2$ mod $(x,y,z,w)$, then $J$ is $(x,y,z,w)$-primary  with  pd$(S/J)=5$ and e$(S/J)=3$.\\
    
    Furthermore\\
     $L=(x^2, y^2, z^2, w^2+ax+by+cz):J=(x^2, y^2, z^2, w^2+ax+by+cz, xyz, w(xy(el-kf) - xz(em-gk) + yz(fm-gl))$ and pd$(S/L)=5$
   \end{customlemma}
         
  \begin{customlemma}{A.11}\label{lemma11}  If   $$J=(x,y,z)^2+(wx,wy,wz,w^2+ ax+by+cz, ex+fy, kx+fz, ky-ez)$$ with ht$(x,y,z,w,e,f,k)=7$, then $J$ is $(x,y,z,w)$-primary with pd$(S/J)=6$ and $e(S/J)=3$.\\
  
  Furthermore, \\
  $L=(x^2, y^2, z^2, w^2+ ax+by+cz):J=(x^2, y^2, z^2, w^2+ ax+by+cz, xyz, xywk + xzwe - yzwf)$ and pd$(S/L)=5$
  \end{customlemma}  
  
    \begin{customlemma}{A.12}\label{lemma12}  If   $$J= (x, y)+ (z,w)^3+ (az+bw )  $$ with ht$(x,y,z,w,a,b)=6$, then $J$ is $(x,y,z,w)$-primary with pd$(S/J)=5$ and $e(S/J)=3$.
  \end{customlemma}
  
        \begin{customlemma}{A.13}\label{lemma13}  If   $$J=(x)+y(y, z, w,)+( z, w)^3+(ay+bz+cw, dy+z^2 )$$ with ht$(x,y,z,w, b,c)=$ ht$(x,y,z,w,c,d)=6$, then $J$ is $(x,y,z,w)$-primary with pd$(S/J)=5$ and $e(S/J)=3$.
  \end{customlemma}   
  
     \begin{customlemma}{A.14}\label{lemma14}  If   $$J=(x)+y(y, z, w,)+( z, w)^3+(ay+bz+cw, dy+zw )$$ with ht$(x,y,z,w,b,c,d)=7$, then $J$ is $(x,y,z,w)$-primary with pd$(S/J)=5$ and $e(S/J)=3$.
  \end{customlemma}

    \begin{customlemma}{A.15}\label{lemma15}  If   $$ J=(x)+y(y, z, w)+( z, w)^3+( ay+bz+cw, cy+z^2, by-zw)$$ or
    $$ J=(x)+y(y, z, w)+( z, w)^3+(ay+bz+cw, cy+zw, by-w^2)$$ or 
    $$J=(x)+y(y, z, w)+( z, w)^3 +(a y+bz+cw, by+zw, cy-z^2 )$$
    
     with ht$(y,z,w, b,c)=5$, then $J$ is $(x,y,z,w)$-primary with pd$(S/J)=5$ and $e(S/J)=3$.
     
  \end{customlemma}

     \begin{customlemma}{A.16}\label{lemma16}  If   $$ J=(x,y)^2+(x,y)(z,w)+ ( ax+by+cz+dw, ex+fy+z^2+\alpha w^2, gx+hy+zw )$$ with ht$(x,y,z,w, c,d)=6$, ht$(c,g,h) \geq 2$, ht$(d,g,h) \geq 2$, ht$(e,f,g,h) \geq 2$ and ht$\left(I_2  \begin{pmatrix}e&f\\ g&h \end{pmatrix}\right) =1$, then $J$ is $(x,y,z,w)$-primary with pd$(S/J)=5$ and $e(S/J)=3$.
  \end{customlemma} 

     \begin{customlemma}{A.17}\label{lemma17}  If  $$J=(x,y)^2+(x,y)(z,w)+( ax+by+cz+dw, ex+z^2+ w^2, ey+zw,c^2y+cdx+d^2y)$$ or $$J=(x^2, xy, xz, xw, y^2, yz, yw,  w^3, ax+by+cz+dw, ex+z^2, ey+zw,cx+dy)$$ with ht$(c,d,e)=3$ mod $(x,y,z,w)$,  then $J$ is $(x,y,z,w)$-primary with pd$(S/J)=5$ and $e(S/J)=3$. 
  \end{customlemma}

    \begin{customlemma}{A.18}\label{lemma18} If $$J=(x,y)^2+(x,y)(z,w)+  ( ax+by+cz+dw, ex+fy+ z^2+\alpha w^2,  cx+zw, dx-z^2 )$$ with ht$(c,d)=$ht$(e,f) =2$ mod $(x,y,z,w)$, then $J$ is $(x,y,z,w)$-primary with pd$(S/J)=5$ and $e(S/J)=3$. 
      \end{customlemma}
      
       \begin{customlemma}{A.19}\label{lemma19} If $$J=(x,y)^2+(x,y)(z,w)+( ax+by+cz+dw, ex+fy, gx+hy+z^2)$$ with ht$(c,d)=$ht$(e,f)=2$ mod $(x,y,z,w)$, then $J$ is $(x,y,z,w)$-primary with pd$(S/J)=5$ and $e(S/J)=3$. 
      \end{customlemma}   
      
           \begin{customlemma}{A.20}\label{lemma20} If $$J=(x,y)^2+(x,y)(z,w)+ (ax+by+cz+dw, ex+fy+w^2, gx+hy+z^2)$$ with ht$(x,y,z,w, c,d)=6$, ht$(c,e,f) \geq 2$, ht$(d,g,h) \geq 2$ and ht$\left( I_2 \begin{pmatrix}e&f\\ g&h \end{pmatrix} \right)=1$ , then $J$ is $(x,y,z,w)$-primary with pd$(S/J)=5$ and $e(S/J)=3$. 
      \end{customlemma}   
      
           \begin{customlemma}{A.21}\label{lemma21} If $$J=(x,y)^2+(x,y)(z,w)+(ax+by+cz+dw, cx+w^2, gx+hy+z^2, dx-zw)$$ with ht$(x,y,z,w, c,d)=6$, ht$(d,g,h) \geq 2$, then $J$ is $(x,y,z,w)$-primary with pd$(S/J)=5$ and $e(S/J)=3$. 
      \end{customlemma}

    \begin{customlemma}{A.22}\label{lemma22}  If  $$J=  (x, y, z,w)^3+ (ax+by+cz+dw, ex+fy+gz+hw, kx+ly+mz+nw)$$ with ht$(x,y,z,w,I_3(N)) \geq 7$ where   $$\tiny N=\left(
\begin{array}{*{14}c}
 0& 0 & 0 & 0 & a & b & c & d & e & f & g & h &\\
 -a& -b & -c & -d & 0 & 0 & 0 & 0 &  k & l &m &n& \\
-e& -f& -g&-h&-k&-l&-m&-n& 0&0&0&0 & \\
\end{array}
\right)$$   then $J$ is $(x,y,z,w)$-primary with pd$(S/J)=6$ and $e(S/J)=3$.
  \end{customlemma}  

   \begin{customlemma}{A.23}\label{lemma23}  If  $$J=  (x, y, z,w)^3+ (ax+by, ay+bz, xz-y^2, kx+ly+mz+nw)$$
   or $$J=  (x, y, z,w)^3+ (ax+by, ay+bz, yz-xw, kx+ly+mz+nw)$$
    with ht$(x,y,z,w,a,b,k,l,m,n) \geq 7$,    then $J$ is $(x,y,z,w)$-primary with pd$(S/J)=6$ and $e(S/J)=3$.
  \end{customlemma}

      \begin{customlemma}{A.24}\label{lemma24}  If  $$J = (x,y,z,w)^3+(ax+fy, ex+fy+gz, kx+ly+mz+nw, z^2a-ywa+ywe-yzk)$$ or 
      $$J=(x,y,z,w)^3+(ax+gy+fz,ex+fy+gz , kx+fz+gw, ya- wa - ye + yk)$$
       with ht$(x,y,z,w,a,e,f,g, k) \geq 8$,  then $J$ is $(x,y,z,w)$-primary with pd$(S/J)=7$ and $e(S/J)=3$.
  \end{customlemma}

    \begin{customlemma}{A.25}\label{lemma25}  If  $$J = (x, y, z,w)^3+ (ax+by, ex+fy+nz, kx+ly+nw, (af-be)w-(al-kb)z))$$ with ht$(x,y,z,w,a,b,n)= 7$ and   ht $(I_5(N)) \geq 3$ mod$(x,y,z,w)$ where 
    $$\tiny N'=\left(
\begin{array}{*{14}c}
 0& 0 & 0 & 0 & a & b & 0& 0 & e & f & n & 0 &0&0\\
 -a& -b & 0& 0 & 0 & 0 & 0 & 0 &  k & l &0 &n &0&0 \\
-e& -f& -n&0&-k&-l&0&-n& 0&0&0&0 &0&0 \\
0&0&af-be&0&0&0&al-bk&0&0&0&el-fk&0&n&0\\
0&0&0&af-be&0&0&0&al-bk&0&0&0&el-fk&0&n\\
\end{array}
\right),$$  then $J$ is $(x,y,z,w)$-primary with pd$(S/J)=6$ and $e(S/J)=3$.
  \end{customlemma}
  
   \begin{customlemma}{A.26}\label{lemma26}  If  $$J = (x, y, z,w)^3+ (ax+by, ex+nz, kx+nw, ew+kz)$$ with ht$(x,y,z,w,a,b,n,e,k) \geq 8$,  then $J$ is $(x,y,z,w)$-primary with pd$(S/J)=7$ and $e(S/J)=3$.
  \end{customlemma}
    
 \begin{customlemma}{A.27}\label{lemma27}  If  $$J =(x,y,z)^3+(w, ax+by, ex+fy+nz) $$ with ht$(x,y,z,w,a,b,n,e,f) \geq 7$,  then $J$ is $(x,y,z,w)$-primary with pd$(S/J)=6$ and $e(S/J)=3$.
  \end{customlemma}

  \begin{customlemma}{A.28}\label{lemma28}  If  $$J =(x,y,z,w)^3+(g_i, g_j, g_k)+(x,y,z,w)g_l$$ where $g_1=ax+by+q_1, g_2=ex+fy+nz+q_2, g_3= kx+ly+nw+q_3,$ and  $g_4=(af-be)w-(al-kb)z) +q_4$, with $q_1, q_2, q_3 \in (z,w)^2$, $q_4 \in (x,y)^2$ and one of the $q_i \neq 0$, ht$(x,y,z,w,a,b,n)= 7$ and   ht $(I_5(N)) \geq 3$ mod $(x,y,z,w)$ where 
    $$\tiny N'=\left(
\begin{array}{*{14}c}
 0& 0 & 0 & 0 & a & b & 0 & 0 & e & f & n & 0 &0&0\\
 -a& -b & 0& 0 & 0 & 0 & 0 & 0 &  k & l &0 &n &0&0 \\
-e& -f& -n&0&-k&-l&0&-n& 0&0&0&0 &0&0 \\
0&0&af-be&0&0&0&al-bk&0&0&0&el-fk&0&n&0\\
0&0&0&af-be&0&0&0&al-bk&0&0&0&el-fk&0&n\\
\end{array}
\right),$$  then $J$ is $(x,y,z,w)$-primary with pd$(S/J)=6$ and $e(S/J)=3$.
  \end{customlemma}  
   
     \begin{customlemma}{A.29}\label{lemma29}  If   $$J =(x, y, z,w)^3+(ax+gy, ex+gz, ey-az, kx+ly+mz+nw)$$ with ht$(x,y,z,w,a,g,e,k,l,m,n) \geq 8$ 
 then $J$ is $(x,y,z,w)$-primary with pd$(S/J)=7$ and $e(S/J)=3$.
  \end{customlemma}

  \begin{customlemma}{A.30}\label{lemma30}  If  $$J =(x,y,z,w)^3+(g_i, g_j, kx+ly+mz+nw)+(x,y,z,w)g_l$$ where $g_1=ax+gy+q_1, g_2=ex+gz+q_2$ and $g_3= ey-az+q_3$ with $q_1 \in (z,w)^2, q_2 \in (y,w)^2, q_3 \in(x,w)^2$,  ht$(x,y,z,w,a,e,g,k,l,m,n) \geq 8$ 
 then $J$ is $(x,y,z,w)$-primary with pd$(S/J)=7$ and $e(S/J)=3$.
  \end{customlemma}  

  \begin{customlemma}{A.31}\label{lemma31}  If  $$J =(xu, yu, zu, wu, xv, yv, zv, wv, xl, yl, zl, wl, ax+by+cz+dw)$$ where  ht$(x,y,z,w,u,v,l)=7$ 
 then $J$ is unmixed with pd$(S/J)=6$ and $e(S/J)=3$.\\
 
 Furthermore, $$L = (xu, yv, zl, ax+by+cz+dw):J= (xu, yv, zl, ax+by+cz+dw, xyz, duvl)$$ and pd$(S/L)=5$.
  \end{customlemma}

\begin{customlemma}{A.32}\label{lemma32}  If  $$J = (xy, xz, xs, xt, y^2, yz,yw, z^2, zw,  w^2t, w^2s,  by+cz+ws, ey+fz)$$ where  ht$(x,y,z,w,s,t,I_2\begin{pmatrix}b&c&d\\ e&f&g \end{pmatrix}) \geq 7$
 then $J$ is unmixed with pd$(S/J)=6$ and $e(S/J)=3$. \\
 Furthermore, \begin{align*} L=&(xt, y^2, z^2,  by+cz+ws):J\\
 =&(xt, y^2, z^2,  by+cz+ws, yzst, xyzw, ts^2(ye-zf))
 \end{align*}
 and pd$(S/L)=5$.
 
  \end{customlemma}  

  \begin{customlemma}{A.33}\label{lemma33}  If  $$J =(ax+by+cz+tw, ex+fy+gz, kx+ly+mz)+(x,y,z,t)\cap(x,y,z,w)^2 $$ where ht$(x,y,z,w,I_2\begin{pmatrix}e&f&g \\ k& l& m \end{pmatrix})=6$,
 then $J$ is unmixed, pd$(S/J)=5$ and $e(S/J)=3$.\\
 Furthermore,\\
 \begin{align*} 
  L&=(x^2,ax+by+cz+tw, ex+fy+gz, kx+ly+mz ):J\\
 &=(x^2,ax+by+cz+tw, ex+fy+gz, kx+ly+mz ,x(gl-fm), t(gl-fm)^2)
 \end{align*} 
  and pd$(S/L)=5$.

  \end{customlemma}  
  
    \begin{customlemma}{A.34}\label{lemma34}  If  $$J = (ax+by+cz+tw, ex+fy+gz)+(x,y,z,t)\cap[(x,y,z,w)^2+(kx+ly+mz+nw)]$$ where  ht$(x,y,z,w,t,I_2\begin{pmatrix}e&f&g\\ k&l&m \end{pmatrix})=7$
 then $J$ is unmixed, pd$(S/J)=5$ and $e(S/J)=3$.\\
Furthermore, 
{\small \begin{align*} L=&(x^2, y^2,z^2, w^2t ):J\\
=&\left(x^2, y^2,z^2, w^2t , xyzyw, xyztI_3 \begin{pmatrix} a&b&c\\e&f&g\\k&l&m \end{pmatrix}-xywtI_3 \begin{pmatrix} a&b&t\\e&f&0\\k&l&n \end{pmatrix}+xzwtI_3 \begin{pmatrix} a&c&t\\e&g&0\\k&m&n \end{pmatrix}-yzwtI_3 \begin{pmatrix} b&c&t\\f&g&0\\l&m&n \end{pmatrix}  \right)
\end{align*}}

and pd$(S/L)=5$.

  \end{customlemma} 
  
     \begin{customlemma}{A.35}\label{lemma35}  If  $$J =(x,y,z)^2 +( wx, wy, wz,w^2t,ax+by, kx+ly+nw, ex+fy+nz)$$ where  ht$(x,y,z,w,a,b,n)=7$  then $J$ is unmixed, pd$(S/J)=5$ and e$(S/J)=3$. \\
     
Furthermore, $$L=(x^2, y^2, z^2, kx+ly+nw):J= (x^2, y^2, z^2, kx+ly+nw, xyzn, xyn(af-be)+zn^2(yb-xz))$$ and pd$(S/L)=5$.
  \end{customlemma}
  
  \begin{customlemma}{A.36}\label{lemma36}  If  $$J =(ax+gy, ex+gz, ey-az, kx+ly+mz+nw, x^2, xy, xz, xw, y^2, yz, yw, z^2, zw, w^2t )$$ where  ht$(x,y,z,w,a,e,g)=7$  then $J$ is unmixed, pd$(S/J)=6$ and e$(S/J)=3$. \\
     
Furthermore, $$L=(x^2, y^2, z^2, kx+ly+nw):J= (x^2, y^2, z^2, kx+ly+mz+nw, xyzn, n^2(xaz+xye-gyz))$$ and pd$(S/L)=5$.
  \end{customlemma}
  
   \begin{customlemma}{A.37}\label{lemma37}  If  $$J =(x,y,z,w,)^3+(ax+by, ay+bz,y^2-xz, kx+ly+mz+nw)$$ or $$J =(x,y,z,w,)^3+(ax+by, az+bw,yz-xw, kx+ly+mz+nw)$$ where $n \neq 0$ and ht$(x,y,z,w,a,b,k,l,m,n) \geq 7$  then $J$ is $(x,y,z,w)$-primary, pd$(S/J)=6$ and e$(S/J)=3$. \\
    If  $$J =(x,y,z,w,)^3+(ax+by, ay+bz,az+bw, y^2-xz, , z^2 - yw, yz - xw)$$ and ht$(x,y,z,w,a,b) \geq 6$  then $J$ is $(x,y,z,w)$-primary, pd$(S/J)=5$ and e$(S/J)=3$.
   
   \end{customlemma}
   
    \begin{customlemma}{A.38}\label{lemma38}  If $$J =(x,y,z,w)^3+(dx+ay, cx+by+az, bz+aw, ywc + z^2 d- ywd) \hspace{0.2cm} \mbox {or}$$
      $$J =(x,y,z,w)^3+(dx+by, cx+by+az, bz+aw, yzc - z^2 d + ywd) \hspace{0.2cm} \mbox {or}$$  $$J =(x,y,z,w)^3+(dx+ay, cx+bz+aw, by+az, y^2c-yzd+z^2d-ywd) \hspace{0.2cm} \mbox {or}$$ $$J =(x,y,z,w)^3+(dx+by, cx+bz+aw, by+az, yzc-z^2d+ywd)\hspace{0.2cm} \mbox {or}$$  $$J =(x,y,z,w)^3+(ax+by, cx+by+az, dx+bz+aw, xzc-ywc-xyd+yzd) \hspace{0.2cm} \mbox {or}$$ $$J =(x,y,z,w)^3+(ax+by, cx+by+az, dx+bz+aw, xzc-zwc-xyd+z^2d) $$
     with ht$(x,y,z,w,a,b,c,d)=8$  then $J$ is $(x,y,z,w)$-primary, pd$(S/J)=7$ and e$(S/J)=3$.
   \end{customlemma}
   
       \begin{customlemma}{A.39}\label{lemma39}  $$ \mbox{If} \hspace{0.2cm} J =(x,y,z,w)^3+(dx+ay, cx+by+az, bz+aw, ywc + yzd -zwd) \hspace{0.2cm} \mbox {or}$$
      $$J =(x,y,z,w)^3+(dx+ay+az, cx+by+az, az+bw, ywc + zwc +yzd -zwd) \hspace{0.2cm} \mbox {or}$$   $$J =(x,y,z,w)^3+(dx+ay+bz, cx+by+az, az+bw, z^2 c -ywc - yzd
  + zwd) \hspace{0.2cm} \mbox {or}$$ $$J =(x,y,z,w)^3+(dx+by, cx+bz+aw, ay+bz, y^2c - yzd + zwd)\hspace{0.2cm} \mbox {or}$$ 
   $$J =(x,y,z,w)^3+(dx+by+bz, cx+bz+aw, ay+bz, y^2c + yzc -yzd+zwd)\hspace{0.2cm} \mbox {or}$$  $$J =(x,y,z,w)^3+(dx+by+az, cx+bz+aw, ay+bz, y^2c-z^2c-yzd+zwd) \hspace{0.2cm} \mbox {or}$$ 
     with ht$(x,y,z,w,a,b,c,d)=8$  then $J$ is $(x,y,z,w)$-primary, pd$(S/J)=7$ and e$(S/J)=3$.
   \end{customlemma}
    
       \begin{customlemma}{A.40}\label{lemma40}  $$ \mbox{If} \hspace{0.2cm} J =(x,y,z,w)^2+(ax+by, kz+lw, ly+az, kx+bw, bz -lx, aw -ky) $$
  with ht$(x,y,z,w,a,b,k,l)=8$  then $J$ is $(x,y,z,w)$-primary, pd$(S/J)=7$ and e$(S/J)=3$.
   \end{customlemma}
   
\section*{Acknowledgment}

\end{document}